\documentclass{amsart}

\usepackage[mathscr]{eucal}
\usepackage{amsmath}
\usepackage{amsthm}
\usepackage[colorlinks]{hyperref}
\hypersetup{
 colorlinks=true,
 citecolor=magenta,
 linkcolor=blue,
 urlcolor=cyan}
\usepackage{enumitem}
\usepackage{tikz}
\usetikzlibrary{matrix,arrows,decorations.pathmorphing}

\newtheorem{theorem}{Theorem}
\newtheorem{lemma}[theorem]{Lemma}

\theoremstyle{definition}
\newtheorem{example}[theorem]{Example}
\newtheorem{remark}[theorem]{Remark}

\newcommand{\Kernel}{\mbox{\rm Ker}\hspace{0.02in}}
\newcommand{\inj}{\mbox{\rm Inj}\hspace{0.02in}}
\newcommand{\sur}{\mbox{\rm Sur}\hspace{0.02in}}
\newcommand{\Cokernel}{\mbox{\rm Coker}\hspace{0.02in}}
\newcommand{\Range}{\mbox{\rm Range}\hspace{0.02in}}
\newcommand{\Index}{\mbox{\rm Index}\hspace{0.02in}}
\newcommand{\buletita}{{\tiny$\bullet$}}

\begin{document}
\title[On global inverse theorems]{On global inverse and implicit theorems\\Revised version}
\author{Olivia Gut\'u}
\date{\today}
\begin{abstract}
Since the Hadamard Theorem, several metric and topological conditions have emerged in the literature to date, yielding global inverse theorems for functions in different settings. Relevant examples are the mappings between infinite-dimensional Banach-Finsler manifolds, which are the focus of this work. Emphasis is given to the nonlinear Fredholm operators of nonnegative index between Banach spaces. The results are based on good local behavior of $f$ at every $x$, namely, $f$ is a local homeomorphism   or $f$ is locally equivalent to a projection. The general structure includes a condition that ensures a global property for the fibres of $f$, ideally expecting to conclude that $f$ is a global diffeomorphism or equivalent to a global projection.  A review of these results and some relationships between different criteria are shown. Also, a global version of Graves Theorem is obtained for a sui\-ta\-ble submersion $f$ with image in a Banach space: given $r>0$ and $x_0$ in the domain of $f$  we give a radius $\varrho(r)>0$, closely related to the hypothesis of the Hadamard Theorem, such that $B_{\varrho}(f(x_0))\subset f(B_r(x_0))$.
\end{abstract}
\maketitle
\tableofcontents

Let $f:\mathbb{R}^n\rightarrow \mathbb{R}^n$ be a differentiable mapping. Consider the nonlinear system $f(x)=y$. In his seminal article \cite{hadamard} of 1906, Hadamard establishes an existence and unicity condition for a nonlinear system $f(x)=y$ in terms of: $$\mu(x)=\min_{v\neq 0}\frac{\|\mbox{J}f(x)v\|_2}{\|v\|_2},$$
where $\mbox{J}f(x)$ is the Jacobian matrix of $f$ at $x$ and $\|\cdot\|_2$ is the Euclidean norm in $\mathbb{R}^n$. If $n=1$, Hadamard points out that if the first derivative is positive at every point $x\in \mathbb{R}^n$ then the nonlinear system has at most one solution, but its existence can not always be assured unless  $\int_{-\infty}^a f'(x)dx=\infty$ and $\int_{a}^\infty f'(x)dx=\infty$. For $n>1$, Hadamard asserts that it is not enough to replace the derivative $f'(x)$ in the above integrals by the Jacobian determinant at $x$, but the appropriate condition for global inversion is: 
$$\int_0^\infty \min_{\|x\|=\rho}\mu(x) d\rho=\infty$$
---referred to in this paper as the {\it Hadamard integral condition}--- provided that $\mu(x)>0$ for all $x\in \mathbb{R}^n$. Hadamard also conjectures the ``properness criterion'' established in the mid thirties by Banach and Mazur \cite{banachmazur} and by Cacciopoli \cite{cacciopoli}, which asserts that a map between Banach spaces is a proper local homeomorphism if and only if  $f$ is a global one. The properness condition has been widely reported and extended in the literature in different frameworks. The same can be said of the Hadamard integral condition and  similar me\-tric criteria in settings where a correct generalization of $\mu(x)$ can be established, sometimes devoted to special cases. The properness criterion was relaxed to closedness by 
Browder \cite{browder} in the context of topological spaces. Besides, the Hadamard Theorem was extended to the infinite-dimensional setting by L\'evy \cite{levy}, who considered the case of smooth mappings between Hilbert spaces. In the late sixties, John \cite{john} also obtained an extension of the Hadamard integral condition for nonsmooth mappings $f$ between Banach spaces in terms of the lower scalar Dini derivative of $f$. For the proof, he used the prolongation of local inverses of $f$ along lines.  Soon after, Plastock \cite{plastock} introduced a limiting property for lines, called {\it condition $L$}, analogous to the continuation property used by Rheinboldt \cite{rheinboldt} in a more abstract context. Plastock proved that a local homeomorphism $f$ satisfies condition $L$ if and only if it is a global homeomorphism. He also showed that the properness, closedness, and the Hadamard integral condition all imply the condition $L$.  Since then, condition $L$ has proved to be quite useful in global inversion theorems. In the same vein, Ioffe \cite{ioffe} extended the Hadamard Theorem in terms of the so-called surjection constant of the mapping $f$ making use of the condition $L$. More recently, this approach was used to give necessary and sufficient conditions for a map $f$ to be a global homeomorphism for a large class of metric spaces with nice local structure ---which includes Banach-Finsler manifolds \cite{oliviajesuscoverings}. As a consequence, an extension of the Hadamard Theorem is obtained in terms of a metric version of $\mu(x)$, a kind of lower scalar Dini derivate. In \cite{oliviajesusisabel} an estimation of the domain of invertibility around a point is provided for  a local homeomorphism between length metric space, inspired by the aforementioned  work of John \cite{john}.  In finite-dimensional case, analogous results were obtained by Pourciau \cite{porciau1, porciau2} by means of the Clarke generalized Jacobian of $f$, recently extended by Jaramillo {\it et al.} \cite{jesusyoscar} for locally Lipschitz mappings between finite-dimensional Finsler manifolds. The above proofs  rely somehow upon a monodromy argument  to  ensure the ``path lifting property'' of the homotopy theory, which finally leads us to conclude that $f$ is a global homeomorphism or more generally, a covering map.  

Since the early nineties, some global inversion conditions that test the effectiveness of the monodromy argument have  emerged. The crucial hypothesis is a variant of the Palais-Smale condition of a suitable function in terms of  $f$. In this regard, Katriel \cite{katriel} also considered the Ioffe surjection constant in order to obtain global inversion theorems in certain metric spaces by methods of critical point theory and based on an abstract mountain-pass theorem. The Katriel's technique has been used as an alternative tool to the monodromy argument to obtain global inverse function theorems in infinite-dimensional case; see for instance \cite{idzak1}. For functions between Euclidean spaces of the same dimension, Nollet and Xavier \cite{xaviernollet} improved the Hadamard integral condition using the Palais-Smale condition but with dynamical systems techniques. This work was recently extended in \cite{xaviernegativecurvature} for finite-dimensional manifolds. 

The monodromy argument also appears in the proof of the Ehresmann Theorem (mid-fifties) \cite{ehresmann} namely, if $X$ and $Y$ are finite-dimensional manifolds with $X$ paracompact and $Y$ connected then a proper submersion (i.e. $df(x)$ is onto for all $x\in X$) $f:X\rightarrow Y$   is a fibre bundle. In this case, we can talk about a ``horizontal path lifting property''. Ehresmann's technique was extended to Banach manifolds by Earle and Eells \cite{eells} by means of a condition in terms of the norm of  a right inverse of $df(x)$. In the late nineties,  Rabier \cite{rabier}  provided a more readily usable condition to ensure the validity of the Earle-Eells criterion. He introduced  the concept of strong submersion with uniformly split kernels for  mappings between Banach-Finsler manifolds, in terms of the corresponding analogue of $\mu(x)$. As  Rabier pointed out in his previous work \cite{rabier3} the strong submersion condition ``interpolates'' two well-known, but until that moment, unrelated hypotheses corresponding to the two extreme cases:  Hadamard's criterion when $f:X\rightarrow Y$ is a local diffeomorphism and the Palais-Smale condition when $Y=\mathbb{R}$. Similar results can be found in \cite{oliviajesusfibrations}. 

The purposes of the current work are the following.  First, to present a survey of such  results in a uniform framework including simpler proofs or adjustments for known theorems. Second, to establish some relationships between different known conditions. Finally, to present some new related results. All of this is done in the context of smooth functions between Banach-Finsler manifolds since these have tools to make the exposition more clear and intuitive, but are also general enough to include two examples of particular importance:  {\it Banach spaces} and  {\it Riemannian manifolds}. 

To this end, in Section 1  we introduce the terms {\it surjectivity} and {\it injectivity indicators} of a function. The surjectivity indicator is the extension of $\mu(x)$ given by Rabier \cite{rabier3,rabier} for Banach-Finsler manifolds. The injectivity indicator is proposed by the author in order to provide context for other global inversion conditions found in the literature. Roughly speaking, if the surjectivity indicator is positive then the function is locally as a projection and if the injectivity indicator is positive then the map is locally as  an embedding. Of course, if both are positive then $f$ is a local homeomorphism.  A brief summary of local inversion for nonlinear Fredholm operator is also presented, as they are natural prospects in this context. Section 2 is devoted to the study of the conditions for a local homeomorphism to be a global one; emphasizing the case where the spaces are Banach spaces and Riemannian manifolds. In the first instance, the purely topological characterizations of global homeomorphisms between Banach manifolds are revisited. A simple proof that the continuation property for minimal geodesics characterize the covering maps between Riemannian manifolds is given. Furthermore, if $f$ is a function with image in a Cartan-Hadamard manifold we show that the domain where the inverse of $f$ is defined is an open star shaped set. The corresponding statements are given in Lemma \ref{cartan-hadamardimages} and Lemma \ref{riemanniancase}, respectively. The rest of the section is devoted to the $C^1$ case. Since for local diffeomorphisms the surjection constant of Ioffe and the lower Dini derivative both coincide with the indicators, the global inversion criteria mentioned above can be expanded or adapted to this framework. We show that these results can be deduced from Earle-Eells criterion for a submersion to be a fibre bundle. For the case where the submersion is a local diffeomorphism this condition characterizes global diffeomorphisms  ---and coincides with the concept of strong submersion of Rabier aforementioned--- provided that the codomain is simply connected. A simplified proof of this fact is presented, see Theorem \ref{firstcharacterization} and Lemma \ref{lemmafinitefibre}. The limits of the monodromy argument are also evidenced, which apparently are not adequate in more refined criteria. Subsequently, the relationships between different conditions are given e.g. in the case of functions between Banach spaces, see Figure \ref{relationshipbanach} and  Lemma \ref{lemaconectionkatrielps}. At the end of the second section, a charaterization of global diffeomorphisms is presented in terms of a weighted version of an Earle-Eells condition. Section 3 deals with the study of the topological and metric conditions provided in Section 2, but for submersions between Banach manifolds. Finally, we obtain a sort of global Graves Theorem in terms of a surjectivity indicator for mappings with uniformly split kernels (Theorem \ref{globalgraves}). Most technical proofs are presented in the appendix.

\section{Preliminary definitions}

\subsection{Banach manifolds}
Let $X$ be a topological space. An {\it atlas} of class  $C^k$ ($k\geq 1$) on $X$ is a collection of pairs $(W_i,\varphi_i)$ ($i$ ranging in some indexing set) satisfying the following conditions: each $W_i$ is a subset of $X$ and the $W_i$ cover $X$; each $\varphi_i$ is a homeomorphism of $W_i$ onto an open subset $\varphi_i(W_i)$ of some Banach space $E_i$, $\varphi_i(W_i\cap W_j)$ is open in $E_i$ for any $i,j$, and
\begin{enumerate}
\item[(*)] the overlap map $\varphi_j\varphi_i^{-1}:\varphi_i(W_i\cap W_j)\rightarrow \varphi_j(W_i\cap W_j)$ is of class $C^k$ for each pair of indices $i,j$.
\end{enumerate}
Each pair $(W_i,\varphi_i)$ is called a {\it chart} of the atlas. Suppose that $\varphi:W\rightarrow W'$ is a homeomorphism onto an open subset of some Banach space $E$. The pair $(W,\varphi)$ is {\it compatible} with the atlas $\{(W_i,\varphi_i)\}$ if each map $\varphi_i\varphi^{-1}$ defined on a suitable intersection as in $(*)$ is a $C^k$ homeomorphism. Two atlases are compatible if each chart of one is compatible with the other atlas. The relation of compatibility between atlases is an equivalence relation. An equivalence class of atlases of class $C^k$ on $X$  defines a structure of $C^k$-manifold on $X$. If all the Banach spaces $E_i$ in some atlas are homeomorphic then we can find an equivalent atlas for which they are all equal,  say to a Banach space $E$. We then say that $X$ is a $C^k$ manifold {\it modeled} in a Banach space $E$, so can be assumed that all charts have image in $E$. If $E$ is a real Banach space then $X$ is said to be a real $C^k$ manifold modeled in $E$.

Let $X$ be a manifold of class $C^k$ modeled in a Banach space $E$ and let $x\in X$. Consider triples $(W,\varphi,w)$ where $(W,\varphi)$ is a  chart at $x$ and $w$ is an element of $E$. Two triples $(W_1,\varphi_1,w_1)$ and $(W_2, \varphi_2, w_2)$ are equivalent if $d\varphi_2\varphi_1^{-1}(\varphi_1(x))w_1=w_2$. An equivalence class of such triples is called a {\it tangent vector} of $X$ at $x$. The set of such tangent vectors is called the {\it tangent space} of $X$ at $x$ and is denoted by $T_xX$. Each chart $(W,\varphi)$ at $x$ determines a bijection of $T_xX$ onto $E$, namely, the equivalence class of $(W,\varphi,w)$ corresponds to the vector $w$. By means of such a bijection it is possible to transport to $T_xX$ the vector space structure of $E$. Of course, this structure is independent of the chart selected.

Let $X$ and $Y$  be two manifolds modeled in $E$ and  $F$, respectively. A map $f:X\rightarrow Y$ is said to be of class $C^k$  ($k\geq 1$) if given $x\in X$ there exists a chart $(W,\varphi)$ at $x$ and a chart $(U,\psi)$ at $y=f(x)$ such that $f(W)\subset U$ and the function $\psi f\varphi^{-1}:\varphi(W)\rightarrow \psi(U)$ is of class $C^k$. The derivative of $f$ at $x$ is the unique linear map $df(x):T_xX\rightarrow T_yY$ having the following property. If $v$ is a tangent vector at $x$ represented by $w\in E$  in the chart $(W,\varphi)$ then $df(x)v$ is  the tangent vector at $y$ represented by $d\psi f\varphi^{-1}(\varphi(x))w\in F$; see  chapter II of \cite{lang} for more details.

\subsection{Finsler metrics}
Let $X$ be a real $C^k$ manifold modeled in a  Banach space $(E,|\cdot|)$. As usual, $TX=\{(x,v):x\in X\mbox{ and } v\in T_xX\}$ is the tangent bundle of $X$.  If $(W,\varphi)$ is a chart of $X$, then there is a local trivialization of the natural projection $\pi:TX\rightarrow X$ over $W$, namely a  bijection $TW=\pi^{-1}(W)\rightarrow W\times E$ which commutes  with the projection on $W$. For every $w\in E$ and $x\in W$, there is a unique pair $(x,v)\in TW$ where $v_x(w)=d\varphi^{-1}(\varphi(x))w$ is the tangent vector at $x$ represented by $w$ in the chart $(W,\varphi)$. A Finsler structure on $TX$ is a continuous map $\|\cdot\|_X:TX\rightarrow[0,\infty)$ such that:

\begin{enumerate}
\item For every $x\in X$ the map $v\mapsto \|\cdot \|_x:=\|(x,v)\|_X$ is an admissible norm for the tangent space $T_xX$. Namely,  for every chart $(W,\varphi)$ at $x$ the map $\|v_x(\cdot)\|_x$ is a norm  equivalent to $|\cdot|$ on $E$.
\item For every $x_0\in X$ and $k> 1$ there exists a chart $(W,\varphi)$ of $X$ at $x_0$ (depending on $k$)  such that for every $x\in W$ and every $w\in E$:
$$k^{-1}\|v_{x_0}(w)\|_{x_0}\leq \|v_{x}(w)\|_x\leq k\|v_{x_0}(w)\|_{x_0}.$$
\end{enumerate}
A {\it Finsler manifold} is a Banach manifold endowed with a Finsler structure on its tangent bundle. This definition corresponds to Palais  \cite[p. 117]{palais}. The relationship of  the above definition with alternative definitions of Finsler manifolds in the literature is presented in Jim\'enez-Sevilla {\it et al.} \cite{jimenezsevilla}. Since all the results  mentioned  here are valid for Finsler manifolds in the Palais sense {\it we shall assume that all Finsler manifolds are in the Palais sense}. Although some results can be applied to more general Finsler manifolds, for the purpose of this work this point is not relevant.

Every paracompact manifold admits a Finsler structure.  If in addition $X$ is modeled in a real separable Hilbert space $(H,\langle\cdot,\cdot\rangle)$ then it admits a Riemannian metric $g$ \cite[p. 175]{lang}. Given a chart $(W,\varphi)$ by means of the local trivialization of $\pi$ we can transport the metric $g$ to $W\times H$. In a local representation this means that for each $x\in W$ we can identify the inner product $g_x$  on  $T_xX$ with a strictly positive operator $A_x:H\rightarrow H$ such that for every $w\in H$, $g_x(v,v)=\langle A_x w,w\rangle$ where $v=v_x(w)$. The metric $g$ is ``smooth'' at $x_0\in X$ in the sense that for every $\epsilon>1$, there is a chart $(W,\varphi)$ of $x_0$ such that  for every $x\in W$ and $w\in E$ \cite{palais}: $g_{x}(v,v)\leq \epsilon^2 g_{x_0}(v,v)$ and $g_{x_0}(v,v)\leq \epsilon^2 g_{x}(v,v)$. In particular, the function $\|(x,v)\|^2=g_x(v,v)$ defines a Finsler structure on $TX$. If $X=(E,|\cdot|)$ is a Banach space then the identity chart can be considered and the function $\|(x,v)\|=|v|$ defines trivially a Finsler structure on $TX=E\times E$. 

Let $X$ be a $C^1$ Finsler manifold. The {\it length} of a  $C^1$ path $\alpha:[a,b]\rightarrow M$ is defined as $\ell(\alpha)=\int_a^b\|\dot \alpha(t)\| dt$. If $X$ is connected, then it is connected by $C^1$ paths and we can define the associated {\it Finsler metric}: $$d_X(x,x')=\inf\{\ell(\alpha):\alpha \mbox{ is a  $C^1$ path connecting $x$ to $x'$} \}.$$ Thereby $(X,d)$ is in particular a metric length space. The Finsler metric is consistent with the topology given in $X$ and is said to be {\it complete} if it is a complete metric space with respect to the metric $d_X$. From now on, unless otherwise noted, {\it we shall assume that all the manifolds are paracompact, at least $C^1$, and without boundary} in order to simplify the arguments. 

\subsection{Injectivity and surjectivity indicators and Fredholm maps.}
Let $X$ and $Y$ be Banach spaces and let $T:X\rightarrow Y$ be a bounded linear operator. 
As it is known, $T$ is one-to-one if:
$$\inj T:=\inf_{|v|=1}|Tv|>0,$$
and $T$ is onto $Y$ if and only if:
$$\sur T:=\inf_{|v^*|=1}|T^*v^*|>0.$$
In this paper the non negative numbers $\inj T $  and $\sur T$ are called {\it injectivity indicator} and {\it surjectivity indicator} of $T$, respectively. By the Bounded Inverse Theorem the operator $T$ is a linear isomorphism if and only if both indicators are positive. In this case $\|T^{-1}\|=\|{T^{-1}}^*\|=\|{T^*}^{-1}\|$ so: $$\inj T =\sur T = \|T^{-1}\|^{-1}.$$
Consider now the {\it linear} system
$
T(x)=y.
$ The dimension of the quotient space $\Cokernel T= Y/\Range T$ provides a number showing the extent to which the above system can fail to have a solution. Besides, the dimension of the $\Kernel T$ provides a number of the extent to which the system can fail to have a unique solution if it has any solution. Recall, $T$ is a {\it Fredholm operator} if $\dim \Kernel T<\infty$ and $\dim \Cokernel T<\infty$.  The {\it index} of the linear map $T$ is the integer: $$\Index T=\dim \Kernel T-\dim  \Cokernel T.$$ If $T$ is a Fredholm operator then $\Range T$ is closed in $Y$ \cite[p. 156]{yuri}. Of course, a desirable situation is when $T$ is invertible; in this case $\dim \Kernel T=0$ and $\dim \Cokernel T=0$ thus $\Index T=0$. 

Now, let $f:X\rightarrow Y$ be a $C^1$ map between  connected Banach manifolds. Consider the {\it nonlinear} system $f(x)=y$. In \cite{smale} Smale introduced a nonlinear version of Fredholm operators in order to establish a infinite-dimensional version of Sard's Theorem.  The function $f$ is a (nonlinear) {\it Fredholm map} if for each $x\in X$ the derivative $df(x):T_xX\rightarrow T_{f(x)}Y$ is a Fredholm operator. The {\it index of $f$} is defined to be the index of $df(x)$ for some $x$. Since $X$ is connected the definition doesn't depend on $x$. For example, a differentiable map $f:\mathbb{R}^n\rightarrow\mathbb{R}^m$ is Fredholm with positive, negative, or zero index if $n>m$, $n<m$, or $n=m$, respectively.

A point $x\in X$ is called a {\it regular point} if $df(x)$ is surjective  and  {\it singular} or {\it critical} if it is not regular. An image of the critical point under $f$ is called {\it critical value} otherwise  {\it regular value}. Note that if  $f^{-1}(y)=\emptyset$ then $y$ is indeed a regular value.

Suppose that $X$ and $Y$ are endowed by a Finsler structure. Let $x\in X$ and $y=f(x)$. The injectivity indicator of $df(x):T_xX\rightarrow T_{y}Y $ can be defined as:
$$\inj df(x)=\inf_{\|v\|_x=1}\|df(x)v\|_{y}$$
Here $\|\cdot\|_x$ represents the Finsler structure of $TX$ restricted to $T_xX$ and $\|\cdot\|_{y}$ is the Finsler structure of $TY$ restricted to $T_{y}Y$.  In the same way, the surjectivity indicator of $df(x)$ can be defined as:
$$\sur df(x)=\inf_{\|v^*\|_{y}=1}\|df(x)^*v^*\|_{x}$$
In this case $\|\cdot\|_x$ and $\|\cdot\|_{y}$ represent the dual norms on $(T_xX)^*$ and  $(T_{y}Y)^*$, respectively.  If $f:X\rightarrow Y$ is a $C^1$ Fredholm map of index $0$ then the following statements are equivalent:
\begin{enumerate}
\item[\buletita] $f$ is a local diffeomorphism at $x$.
\item[\buletita] $\sur df(x)>0$.
\item[\buletita] $\inj df(x)>0$.
\end{enumerate}
Furthermore, if one of these statements is true then: $$\sur df(x)=\inj df(x)=\|df(x)^{-1}\|^{-1}.$$
Indeed, let $T=df(x)$. Since $X$ and $Y$ are Finsler manifolds then $(T_xX,\|\cdot\|_x)$ and $(T_yY,\|\cdot\|_y)$ are both Banach spaces and $T:T_xX\rightarrow T_yY$ is a linear Fredholm map of index $0$. If $\sur T>0$ then $T$ is onto. Therefore $T$ is injective, so $\inj T>0$. In the same way we conclude that $\inj T>0$ implies $\sur T>0$. The equivalences follow from  the Inverse Mapping Theorem. Note that {\it $f$ is a local diffeomorphism between connected manifolds if and only if it is a Fredholm map of index $0$ without critical points.} 

\subsubsection{Invariance of domain property}\label{sectioninvariancedomain} The classical Brouwer Theorem on invariance of domains states that if $U\subset\mathbb{R}^n$ is an open set and  $f:U\rightarrow\mathbb{R}^n$ is a continuous injective map then $f(U)$ is an open set in $\mathbb{R}^n$. Consequently, a locally injective mapping from $\mathbb{R}^n$ to $\mathbb{R}^n$ is a local homeomorphism.
However, in general this is not longer true  for mappings between infinite-dimensional Banach spaces \cite{krikorian}. Nevertheless, an important consequence from the degree theory  is the following: {\it Let $f$ be a Fredholm map of index $0$ between connected Banach manifolds. If  $f$ is a locally injective map then it  is an open map.} In particular, $f$ is a local homeomorphism. So an ``invariance-of-domain'' property holds for these operators even in infinite dimensions  \cite{tromba}.  This result is applicable even if there are certain types of critical points. 

On the other hand, the ``invariance of domain''  property can be extended to locally compact perturbations of nonlinear Fredholm maps of index $0$. More precisely: if $X$ and $Y$ are Banach spaces, $U$ is an open subset of $X$, $f+k:U\rightarrow Y$ is a locally injective map where $f$ is a Fredholm map of index $0$ and $k$ is continuous locally compact  function, then $f+k$ is an open map \cite{calamai}. A standard argument of composition with charts is sufficient to check that this result holds if  $f + k$ is defined on a connected smooth Banach manifold with image in a Banach space. A point to be considered is that differentiability is not required for $k$. We recall that a map between two topological spaces is {\it locally compact} if any point in its domain has a neighborhood whose image has compact closure. Therefore,  the Schauder's theorem on invariance of domain for compact perturbation of the identity is a special case of the property  above \cite{schauder}. See also \cite{kim, mawhin} and references therein.

As might be expected, it should be noted that these theorems are only valid for the index-zero case. Indeed, if $f$ is a Fredholm map with negative index then the image of $f$ has empty interior. On the other hand, there are no locally injective Fredholm maps with a positive index.

\section{Global homeomorphism theorems}

\subsection{The path lifting property revisited}\label{topologicalcharacterization}
Keep in mind again the equation $f(x)=y$ now with $f:X\rightarrow Y$ a local homeomorphism   between  topological spaces. An important issue of algebraic topology is the ``lifting problem''. Let $I=[0,1]$ and let $p:I\rightarrow Y$ be a continuous path in $Y$ such that $p(0)\in f(X)$. The lifting problem for $f$ is to determine whether there is a continuous path $q:I\rightarrow X$, so-called {\it lifting} of $p$, such that $f\circ q=p$. If it so, it is said that $f$ {\it lifts} the path $p$. Furthermore, $f$ is said to have the {\it  path lifting property} if:
\begin{enumerate}[label={\bf (C1)}]
\item \label{pathlifing} $f$ lifts every continuous path in $Y$ with starting point at $f(X)$.
\end{enumerate} 

 The path lifting property is directly related to the global behavior of a set of solutions of $f(x)=y$ for all $y\in Y$. For example, if $Y$ is path-connected then the path lifting property implies at once that  $f$ is onto, so the nonlinear system always has a solution. In this vein, in the mid-fifties Browder  \cite{browder} established a remarkable result in the general context of Hausdorff topological  spaces $X$ and $Y$ with extra suitable local connectedness  and separation conditions ---including paracompact Banach manifolds--- which asserts that if a local homeomorphism $f:X\rightarrow Y$ has the path lifting property then it is a {\it covering projection}, namely,  $f$ is onto and for each point $y\in Y$ there exists a neighborhood $V$ of $y$ such that $f^{-1}(V)$ is the union of a disjoint family of open sets of $X$, each of which is mapped homeomorphically onto $V$ by $f$. The space $X$ is called the {\it covering space} and the space $Y$ the {\it base space}. The converse of the Browder result is also true, since every covering projection has the path lifting property; see Section 2.2 of \cite{spanier}. If $f$ is a covering projection then the set of the solutions $f^{-1}(y)$ of the nonlinear problem $f(x)=y$  ---that is, the {\it fibre} over $y$--- is a discrete set and all the fibres are homeomorphic if $Y$ is path-connected \cite[p. 73]{spanier}, so we can speak of {\it the} fibre of $f$. 
 
Via the path lifting property, Browder proved that {\it a closed local homeomorphism is a covering map and each fibre of $f$ is a finite set}. This last statement in italics is usually referred as the Browder Theorem.  Recall that, a function $f:X\rightarrow Y$ is said to be {\it closed} if:
\begin{enumerate}[label={\bf (C2)}]
\item \label{closed} The image of any closed set of $X$ is closed in $Y$. 
\end{enumerate}  
The idea of the proof lies in the fact that if $f$ is a local homeomorphism and $p$ is a path in $Y$ beginning at $f(x_0)$ for some point $x_0\in X$:
\begin{enumerate}
\item[\buletita] there is at most one lifting of  $p$ beginning at $x_0$;
\item[\buletita] there always exists a lifting of $p$ locally.
\end{enumerate}
If $f$ is closed then the local lifting of $p$ can be extended to whole interval $I=[0,1]$, so $f$ lifts the path $p$. Obviously not every covering map is a closed map. However, Browder gives a characterization (Theorem 5 of \cite{browder}) of the covering maps in terms of the following condition, called in the current paper the {\it Browder condition}:
\begin{enumerate}[label={\bf (C3)}]
\item \label{browder} For every $y\in Y$ there exists a neighborhood $V$ such that $f$ is a closed mapping on each component of $f^{-1}(V)$ into $V$.
\end{enumerate}

In the late sixties, the article of Rheinboldt \cite{rheinboldt} appeared in the literature where a general theory  is established for global implicit function theorems in terms of the  {\it continuation property}. The continuation property was influenced in part by the so-called continuation method in numerical analysis. A local mapping relation \cite[p. 184]{rheinboldt} ---e.g.  a continuous map--- $f:X\rightarrow Y$ is said to have the continuation property for a subset $P$ if:
\begin{enumerate}[label={\bf (C4)}]
\item \label{rheinboldtc} For any $p\in P$ and any local lifting $q$ of $p$ defined on $[0,\varepsilon)\subset [0,1]$ there exists an increasing sequence $t_n\rightarrow \varepsilon$ such that $\{q(t_n)\}$ converges in  $X$ (in a topological sense).
\end{enumerate}
As Rheinboldt pointed out, this definition represents a simple modification of the path lifting property for  paths in the set $P$. It is therefore not surprising that in fact they are equivalent \cite[p. 185]{rheinboldt}. In particular, the continuation property guarantees the existence of the lifting defined on whole interval $I=[0,1]$.
Various applications to global implicit and inversion theorems are presented in \cite{rheinboldt} essentially for normed linear spaces where $P$ is taken as the  set of all smooth paths on $Y$. As discussed below, in this case much more can be said.

Relatively recently, Gut\'u and Jaramillo \cite{oliviajesuscoverings} presented an extension of some well known results in the framework of metric spaces where the continuation property plays a central role. From  Example 2.2 and Theorem 2.6 of \cite{oliviajesuscoverings} we can deduce that if $Y$ is a connected $C^k$ Banach manifold ---equivalently $C^k$ path-connected \cite[p. 118]{palais}--- then it is enough to consider the set $P$ as the set of the $C^k$ paths in a connected $Y$. More precisely, let $X$ and $Y$ be Banach manifolds, assume $Y$ is connected of class $C^k$ where $1\leq k\leq\infty$,  if $f:X\rightarrow Y$ is a local homeomorphism then following statements are equivalent:
\begin{enumerate}
\item[\buletita]  $f$ has the continuation property for the set of all  $C^k$ paths in $Y$.
\item[\buletita] $f$ is a covering map.
\end{enumerate}

In order to make a clear connection with later results and because the idea is quite simple, we present below the sketch of the proof. The demonstration is based on the ideas in \cite{plastock}, in turn based on the theory of covering spaces of differential geometry.

\

\noindent{\sl Sketch of the proof.} Let $y\in Y$. There exists a  $C^k$ path $p$ joining  some point in $f(X)$ to $y$ since $Y$ is connected and of class $C^k$. Therefore, since $f$ is a local homeomorphism there exists a local lifting $q$ of $p$. As $f$ has the continuation property for the $C^k$ paths then the local lifting can be extended to whole $I=[0,1]$ and $f(q(1))=y$. So, $f$ is onto. Let $(V_y,\psi)$ be a $C^k$ chart centered at $y$ i.e. $\psi(V_y)$ is an open ball in a Banach space centered at $0=\psi(y)$. For every $z\in V_y$ there exists a line segment relative to $\psi$, $p_z(t)=\psi^{-1}(t\psi(z))$, which is a $C^k$ path joining $y$ to $z$. For any $u\in f^{-1}(y)$, as before, there exists lifting $q_z$ of $p_z$ starting at $u$ defined in whole $I$. The lifting is unique since $f$ is a local homeomorphism. The continuity of the map $(t,z)\mapsto p_z(t)$ implies that the sets $O_u=\{q_z(1):z\in V_y\}$ form a disjoint family of open sets  such that  $f^{-1}(V_y)=\bigcup_{u\in f^{-1}(y)}O_u$  and each $O_u$ is homeomorphic onto $V_y$ by $f$. See Remark \ref{conexidadimplicasobre}.

\

Suppose that  $X$ and $Y$ are $C^1$ Banach manifolds, not necessarily connected. It is well known that if $f$ is a closed mapping and $f^{-1}(y)$ is compact then $f$ is a proper map \cite[p. 119]{lee}. Recall, a function between topological spaces $f:X\rightarrow Y$ is said to be {\it proper} if:
\begin{enumerate}[label={\bf (C5})]
\item\label{proper} The preimage of each compact set in $Y$ is compact in $X$.
\end{enumerate}
On the other hand, in this context, every proper map $f:X\rightarrow Y$ is closed \cite{palais2}. Besides, any constant map over $\mathbb{R}$ gives us a simple example of a non-proper closed map. However, for connected manifolds, if  $f$ is a  Fredholm closed map and $\dim X=\infty$ then $f$ is also a proper map \cite{smale}. Moreover, if $X$ and $Y$ are both infinite-dimensional then every continuous non-constant closed map is proper \cite{sabyrkhanov}. Nevertheless, according to the Browder Theorem, regardless of the dimension and connectedness of $X$, {\it  if $Y$ is connected and  $f$ is a local homeomorphism then $f$  is a closed map if and only if it  is a proper map}. In this case,  $f$ is a covering projection with finite fibre. In order to close the circle, note that {\it if $f$ is a covering projection with finite fibre then it is a closed map}.  

By the above arguments, a local homeomorphism $f$  is a covering map provided it is {\it weakly proper} map, namely:
\begin{enumerate}[label={\bf (C6)}]
\item \label{weaklyproper} For every compact  $K\subset Y$ each  component of $f^{-1}(K)$ is compact in $X$.
\end{enumerate}
In fact, if it is so, we can also apply a standard monodromy argument to get a global lifting from the local one. Clearly, condition \ref{weaklyproper} does not characterize the covering projections, for example $f(t)=\exp(2\pi it)$ is a covering map but it is not weakly proper. 

\begin{remark}
It is well known that if a continuous map $f$ is a covering projection then it induces a monomorphism $f_\star:\pi_1(X)\rightarrow\pi_1(Y)$ defined by $f_\star[\omega]=[f\omega]$. So a covering projection $f$ is a homeomorphism if and only if $f_\star\pi_1(X)=\pi_1(Y)$ \cite[p. 77]{spanier}. This occurs for example if $Y$ is simply connected. In this case, if $df(x)$  is a linear isomorphism for all $x\in X$ then $f$ is a global diffeomeomorphism onto $Y$. The reader can consult \cite[p. 47]{ambrosetti} for a direct and elementary proof ---avoiding passing through the monomorphism $f_\star$--- of the so called {\it mo\-no\-dromy theorem}. This classical result relates to the claim that every proper local homeomorphism $f:X\rightarrow Y$ between metric spaces is a global homeomorphism if  $X$ arcwise connected and $Y$ simply connected.  
\end{remark}

Summing up, let $1\leq k\leq\infty$ and let $X$ and $Y$ be Banach manifolds, assume $Y$ is of class $C^k$ and connected, if $f:X\rightarrow Y$  is a local homeomorphism then the following conditions are equivalent to $f$ being a covering projection:
\begin{itemize}
\item[\buletita] $f$ satisfies the Browder condition.
\item[\buletita] $f$ has the path lifting property.
\item[\buletita] $f$ has the continuation property for the set of all  $C^k$ paths in $Y$.
\end{itemize} 
Furthermore, the following conditions are equivalent to $f$ being a covering projection {\it with finite fibre}:
\begin{enumerate}
\item[\buletita] $f$ is a closed map.
\item[\buletita] $f$ is a proper map.
\end{enumerate} 
Finally, $f$ is a covering map provided that:
\begin{enumerate}
\item[\buletita] $f$ is a weakly proper map.
\end{enumerate}
If $Y$ is simply connected then all conditions above are equivalent to $f$ being a global homeomorphism.  The special cases are detailed in the next subsection.

\subsubsection{Banach spaces, Riemannian and Cartan-Hadamard Finsler manifolds}
Independently of Rheinboldt \cite{rheinboldt}, Plastock  \cite{plastock} introduced the {\it condition $L$} for local homeomorphism $f:X\rightarrow Y$ between Banach spaces. The condition $L$  is just the continuation property for the subset $P$  of all the lines  $l_z(t)=y(1-t)+zt$ in $Y$. Plastock sets that {\it a local homeomorphism between Banach spaces is a global one if and only if it satisfies condition $L$}. The proof is a special case of the sketch of the proof given above with $(V,\psi)=(B_r(y),{\rm id_Y}-y)$ for some $r>0$.  

Plastock's result may even be improved, since it is enough to lift only the line segments from a fixed point in the image of $f$. Indeed, following the ideas of John  \cite{john}, suppose that $y_0=f(x)$ for some $x\in X$. If $f$ is a local homeomorphism then there exists a neighborhood $V$ of $y_0$ and an  inverse  defined on $V$. This local inverse can be continued, as far out as possible, by a monodromy process determined by the uniqueness of the continuations. So, a global inverse $f^{-1}_{x}(z):=q_z(1)$ can be constructed on a maximal star  $S_{y_0}$ defined as the set of all $z\in Y$ for which there exists a lifting $q_z$, starting at $x$, of the line  $p_z$ joining $y_0$ to $z$. As John proved, the set $S_{y_0}$ is open and the mapping $f^{-1}_{x}$ is an inverse of $f$ with domain $S_{0}$. Therefore, {\it $f$ is a global homeomorphism if and only $S_{y_0}=Y$ if and only if $f$ lifts the lines $l_z(t)=(1-t)y_0+tz$ for all $z\in Y$}. In this vein, a map $f$ is said to be {\it ray-proper} at $y_0$ if:
\begin{enumerate}[label={\bf (C7)}]
\item \label{rayproper} The pre-image of  the line joining $y_0$ and $z$ is compact for any $z\in Y$ and some $y_0\in Y$. 
\end{enumerate}
Note that a local homeomorphism ray-proper  at some $y_0\in f(X)$ clearly satisfies condition \ref{rheinboldtc} for the set $P$ of rays from $y_0$. A continuous map $f:X\rightarrow Y$ between Banach spaces is proper if and only if it is closed and ray-proper. Furthermore, it not difficult to construct a  ray-proper map which is not proper \cite[pp. 68--72]{nonlinearspectraltheory}. But, if $f$ is a local homeomorphism between Banach spaces e.g. a locally injective Fredholm map of index $0$, then the following conditions are equivalent:
\begin{enumerate}
\item[\buletita] $f$ is ray-proper map at $y_0\in f(X)$.
\item[\buletita] $f$ is a closed map.
\item[\buletita] $f$ is a proper map.
\item[\buletita] $f$ satisfies condition $L$.
\item[\buletita] $f$ has the continuation property for  all lines from $y_0$.
\item[\buletita] $f$ is a global homeomorphism.
\end{enumerate}
 
\begin{remark}\label{wazewski}
Note that if $f$ is a local diffeomorphism then the liftings of the lines $l_w=f(x)+tw$ are given by  the flow of the differential equation $\dot q(t)=df(q(t))^{-1}w$ ---namely, Wa\.zewski equation---  with the initial condition $q(0)=x$ cf.  \cite{demarco}. 
\end{remark}

We claim that the ideas of John can be used to obtain this result  for Cartan-Hadamard manifolds. Recall, in a traditional sense, a Cartan-Hadamard manifold is a Riemannian manifold $(Y,g)$ which is complete, simply connected and with semi-negative curvature; see \cite[p. 235]{lang} for a precise definition.  Because $Y$ is complete, the exponential map $\mbox{exp}_y$ is defined on all $T_yY$ for all $y\in Y$. By the Cartan-Hadamard Theorem the function  $\mbox{exp}_{y}:T_{y}Y\rightarrow Y$ is a global diffeomorphism.  The definition of semi-negative  curvature and the Cartan-Hadamard Theorem can both be extended for some finite and infinite-dimensional Finsler manifolds, namely  {\it Finsler manifolds with spray}, according to the Neeb definition \cite{neeb} which includes: Banach spaces, Riemannian manifolds, and the finite-dimensional Finsler manifolds called Berwald spaces.  In a Cartan-Hadamard manifold  every two points can be joined by a {\it unique} minimal geodesic \cite{neeb}. That is, every Cartan-Hadamard manifold is a {\it uniquely geodesic} space. Surprisingly, not every Banach space is a Cartan-Hadamard manifold since a Banach space is uniquely geodesic if and only if its unit ball is strictly convex; see Proposition 1.6 of \cite{bridson}. For example, $\ell_1$ or $\ell_\infty$ can not be Cartan-Hadamard manifolds.

Let $f:X\rightarrow Y$ be a local homeomorphism. Assume that $X$ is a Banach manifold and  $Y$ is a Cartan-Hadamard Finsler manifold. Let $y_0\in f(X)$ and let $p_z$ be the unique minimizing geodesic segment joining $y_0$ to $z$ in $Y$. As before,  $S_{y_0}$ is the star with vertex $y_0$ defined as the set of all $z\in Y$ for which there is a lifting $q_z$ of $p_z$ such that $q_z(0)=x$. Let $f_x^{-1}(z):=q_z(1)$. The set {\it $S_{y_0}$ is open and the mapping $f_x^{-1}$ is an inverse of $f$ with domain $S_{y_0}$}; see Lemma \ref{cartan-hadamardimages} in the appendix. In particular, the following statements are equivalent:
\begin{enumerate}
\item[\buletita] $f$ has the continuation property for  all minimal geodesics from $y_0$.
\item[\buletita] $f$ is a global homeomorphism.
\end{enumerate}
Finally, if $(Y,g)$ is a Riemannian manifold then for every $y\in Y$ there exists $r>0$ sufficienty small such that $\exp_y:B_g(0,r)\rightarrow B_g(y,r)$ is a diffeomorphism. So if we proceed as before, but with  the paths $p_z(t)=\exp_y(t\exp_y^{-1}(z))$ we can deduce the following fact. Let $f:X\rightarrow Y$  be a local homeomorphism between Banach manifolds, assume $Y$ is Riemannian and connected, thus the following statements are equivalent:
\begin{enumerate}
\item[\buletita] $f$ has the continuation property for the set of minimal geodesics in $Y$.
\item[\buletita] $f$ is a covering map.
\end{enumerate}
This fact can be generalized to manifolds admiting a definition of an exponential map such that $\exp_y$ is a local diffeomorphism  at $0\in T_yY$, an unclear fact in the infinite-dimensional Finsler setting \cite[p. 120]{neeb}. See proof of Lemma \ref{riemanniancase} in the appendix for details.

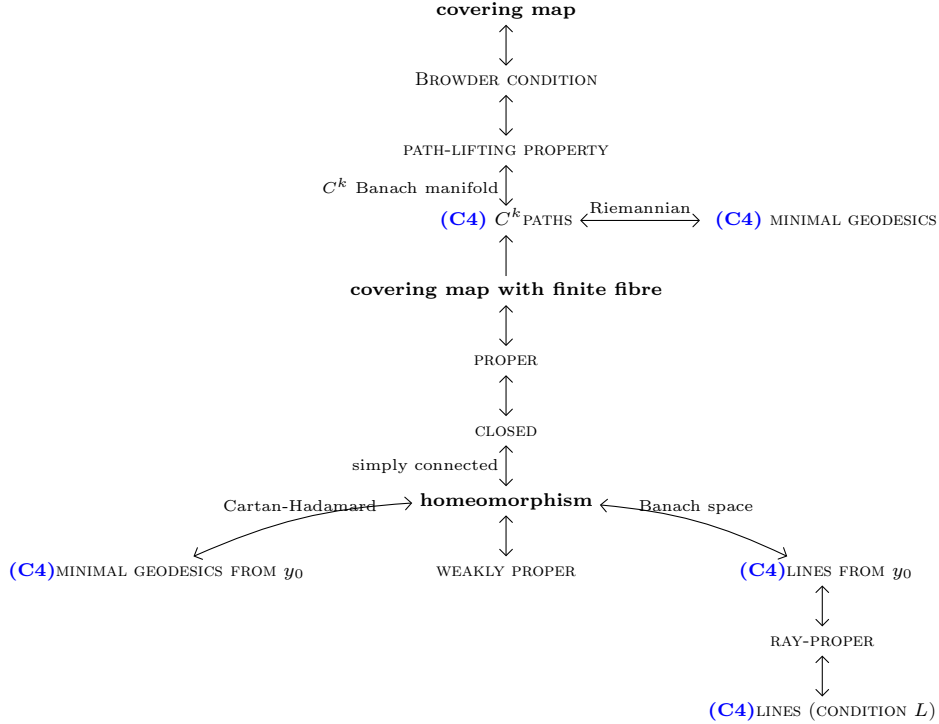
\begin{figure}[htbp]
\begin{center}
{\scriptsize
\begin{tikzpicture}[>=angle 90]
\matrix(m)[matrix of math nodes,
row sep=2em, column sep=1.5em,
text height=1.5ex, text depth=0.25ex]
{ &\mbox{\bfseries covering map} &\\
  &\mbox{\sc Browder condition} &\\
  &\mbox{\sc path-lifting property} &\\
  &\mbox{\ref{rheinboldtc}} ~C^k \mbox{\sc paths} & \mbox{ \ref{rheinboldtc} \sc minimal geodesics}  \\
  &\mbox{\bf covering map with finite fibre}&\\
  &\mbox{\sc proper} &\\
  &\mbox{\sc closed} &\\
  &\mbox{\bf\bfseries homeomorphism}&\\
  \mbox{\ref{rheinboldtc}\sc minimal geodesics from }y_0 & \mbox{\sc weakly proper} & \mbox{ \ref{rheinboldtc}\sc lines from } y_0\\
  &&\mbox{\sc ray-proper}\\
  &&\mbox{\ref{rheinboldtc}\sc lines (condition $L$)}\\
  };
\path[<->] (m-1-2) edge (m-2-2);
\path[<->] (m-2-2) edge (m-3-2);
\path[<->,font=\tiny]
(m-3-2) edge node[left]{$C^k$ Banach manifold} (m-4-2);
\path[<->,font=\tiny]
(m-4-2) edge node[above]{Riemannian} (m-4-3);
\path[<-,font=\tiny]
(m-4-2) edge (m-5-2);
\path[<->](m-5-2) edge (m-6-2);
\path[<->](m-6-2) edge (m-7-2);
\path[<->,font=\tiny]
(m-7-2) edge node[left]{simply connected} (m-8-2);
\path[<->,font=\tiny]
(m-8-2) edge [bend right=10] node[above]{Cartan-Hadamard} (m-9-1);
\path[<->,font=\tiny]
(m-8-2) edge [bend left=10] node[above]{Banach space} (m-9-3);
\path[<->](m-9-3) edge (m-10-3);
\path[<->](m-10-3) edge (m-11-3);
\path[<->](m-8-2) edge (m-9-2);
\end{tikzpicture}
}
\end{center}
\caption{Equivalences and implications for a local homeomorphisms $f:X\rightarrow Y$ between  Banach manifolds, $Y$ connected and also satisfying the condition in the labeled arcs.}
\label{relationshipcoverings}
\end{figure}

\subsection{Metric conditions}
Let  $X$ and $Y$ be Banach spaces and let $f:X\rightarrow Y$ be a (linear) {\it global isometry} namely, $|f(x)|=|x|$ for all $x\in X$.  The isometry $f$ is a distance preserving map and evidently one-to-one. A pertinent question is: Under which conditions $f$ is onto? For example, if the isometry $f$ is a Fredholm map of index zero we may expect that $f$  be onto, since in this case $f(X)$ is always open in $Y$. In this sense, Rheinboldt \cite{rheinboldt} proved that if  $X$ is complete and $f:X\rightarrow Y$ is a continuous map such that $f(X)$ is open in $Y$ and  $|f(u)-f(x)|\geq \alpha |u-x|$ for all $x,u\in X$ and some $\alpha>0$ then $f$ has the continuation property for smooth paths. So, we can conclude that $f$ is surjective.

Suppose that $X$ and $Y$ are connected Finsler manifolds with Finsler distance $d_X$ and $d_Y$, respectively. The Rheinboldt's result invites us to consider the following condition. We shall say that $f$ is an {\it expansive map} if:
\begin{enumerate}[label={\bf (C8)}]
\item \label{expansive}$X$ is complete and $d_Y(f(u),f(x))\geq \alpha~ d_X(u,x)$ for all points $x,u\in X$  and for some $\alpha>0$. 
\end{enumerate}
Let  $f:X\rightarrow Y$  be an expansive Fredholm map of index zero. We can repeat the Rheinboldt technique: Clearly, $f$ is injective and therefore a local homeomorphism, since it is a Fredholm map of index zero. We will check that $f$ has the continuation property for $C^1$ paths. Let $p$ be a $C^1$ path in $Y$, let $q$ be a local lifting of $p$ defined on $[0,\varepsilon)\subset I$, and let $\{t_n\}$ be a sequence converging to $\varepsilon$. The sequence  $\{q(t_n)\}$  is a Cauchy sequence because $\{p(t_n)\}$ converges and $d_Y(p(t_i),p(t_j))\geq \alpha~ d_X(q(t_i),q(t_j))$ for all $t_i,t_j\in \{t_n\}$. The completeness of $X$ implies that $\{q(t_n)\}$ converges in $X$. Therefore $f$ is a covering map with a singleton fibre. So, $f$ is a global homeomorphism. In other words: {\it An expansive Fredholm map of index zero is a global homeomorphism}.

An illustrative example is presented below in order to introduce the relationship between the expansive maps and the surjectivity and injectivity indicator. 

\begin{example}
Let $X$ be a Hilbert space. A {\it strongly monotone operator} $f:X\rightarrow X$ is a map such that: 
\begin{enumerate}[label={\bf (C9)}]
\item \label{monotone}$|\langle f(u)-f(x),u-x\rangle|\geq\alpha|u-x|^2 \mbox{ for all } x, u\in X  \mbox{ and for some }\alpha>0.$
\end{enumerate}
Zarantonello \cite{zarantonello} proved that a strongly monotone operator is a bijective homeomorphism. This is a typical example of an expansive map. Let $x\in X$ and let $h$ be a small enough positive number such that $f(x+h)=f(x)+df(x)(h)+r_x(h)$.  If $h=\epsilon v$ then for some $v\in X$, $|\langle df(x)v,v\rangle|+o(1)\geq \alpha|v|^2$. Hence, $|\langle df(x)v,v\rangle|\geq \alpha|v|^2 \mbox{ and }|\langle v,df(x)^*v\rangle|\geq \alpha|v|^2$ for all $x,v\in X$. Therefore  $\inj df(x)>\alpha$ and $\sur df(x)>\alpha$ for all $x\in X$. 
\end{example}

Condition \ref{expansive} is quite strong, but the Rheinboldt technique  invites us to consider a ``local expansive property'' and the above example in turn leads us to think about the surjectivity and injectivity indicators. In fact, if $X$ and $Y$ are complete and connected Finsler manifolds and $f:X\rightarrow Y$ is a local diffeomorphism then there is the following relationship between the Finslerian distance and the surjectivity and injectivity indicators:
$$D_x^-f:=\liminf_{u\rightarrow x}\frac{d_Y(f(u),f(x))}{d_X(u,x)}=\|df(x)^{-1}\|^{-1}=\sur df(x)=\inj df(x).$$
The first  equality was noted by John \cite{john} for Banach spaces. For Riemannian and Finsler manifolds the proof can be consulted in  \cite{oliviajesuscoverings} and \cite{jesusyoscar}, respectively. So,  for local diffeomorphisms {\it condition} \ref{expansive} {\it implies that both indicators are uniformly bounded below on $X$}. That is, the {\it Hadamard-Levy condition} is fulfilled:

\begin{enumerate}[label={\bf (C10)}]
\item \label{hadamard-levy} $X$ is complete and there is $\beta>0$ such that $\|df(x)^{-1}\|\leq\beta\mbox{ for all } x\in X.$
\end{enumerate}

\subsubsection{The Earle-Eells  condition}
Our goal now is to introduce a fairly general condition in order to get a covering map in terms of the injectivity and surjectivity indicators. In the rest of the section, the symbol  $\mu(x)$  will be used to denote any of the two indicators of $df(x)$. For example, $\mu(x)>0$ means $\inj df(x)>0\mbox{ or }\sur df(x)>0.$ If this is the case and $f$ is a Fredholm operator of index zero then actually both indicators of $df(x)$ are positive (if one indicator is positive, so is the other) and $f$ is indeed a local diffeomorphism and $\mu(x)=\|df(x)^{-1}\|^{-1}$. It is important to note that we don't need the inverse of $df(x)$ to calculate $\mu(x)$ and that can be done via any of the two indicators.

Recall the following direct consequence of the chain rule: if $f$ is a local diffeomorphism and $q$ is a $C^1$ local lifting of a rectifiable path $p$ in $Y$ then: \begin{equation}\label{meanvalue}\ell(q)\cdot \inf\{ \mu(x): x\in \mbox{\rm  image of } q\}  <\ell(p).\end{equation}
So, the sketch of the proof presented in Section \ref{topologicalcharacterization} suggests that the following condition be considered ---called in this paper the {\it Earle-Eells condition}--- in order to get a covering map:
\begin{enumerate}[label={\bf (C11)}]
\item \label{earleell} $X$ is complete and for every $y\in Y$ there exists $\alpha>0$ and a neighborhood $V$ of $y$ such that $\mu(x)\geq\alpha \mbox{ for all } x\in f^{-1}(V).$
\end{enumerate}
Let $f:X\rightarrow Y$ be a Fredholm map of index zero that satisfies condition \ref{earleell}. Because $\mu(x)>0$ for all $x\in X$, as we pointed out before, then the map $f$ is a local diffeomorphism. The condition \ref{earleell} implies that the length of any local lifting of a line segment relative to a chart is rectifiable. So, the local lifting can be extended to whole interval $I=[0,1]$; which leads to the conclusion that $f$ is  a {\it smooth covering map} i.e. $f$ is onto and for each point $y\in Y$ there exists a neighborhood $V$ of $y$ such that $f^{-1}(V)$ is the union of a disjoint family of open sets of $X$, each of which is mapped diffeomorphically onto $V$ by $f$. Therefore,  {\it if $X$ and $Y$ are connected Finsler manifolds and $f:X\rightarrow Y$ is a Fredholm map of index zero then $f$  is a smooth covering map provided $f$ satisfies the Earle-Eells condition}. This result is a special case of  Proposition C of \cite{eells}. Although Earle and Eells request the surjectivity of $f$, this condition can be replaced by connectedness of $Y$. They also request extra and unnecessary smoothness requirements. A simpler proof is presented in the appendix with recently presented ideas (Theorem \ref{firstcharacterization}). Furthermore, it is proven that every smooth covering map with finite fibre satisfies the Earle-Eells condition; see Lemma \ref{lemmafinitefibre} in the appendix. Specially we have: {\it Let $X$ and $Y$ be connected Finsler manifolds. Assume $X$ is complete, $Y$ is simply connected, and $f:X\rightarrow Y$ is a Fredholm map of index zero.  Then the Earle-Eells condition is necessary and sufficient for $f$ to be a global diffeomorphism.}

\begin{remark}{\it Length-path conditions.}
In his seminal paper \cite{hadamard}, Hadamard suggests that if his integral condition is satisfied then ``a path  of infinite length drawn in $X$ can't have an image in $Y$ with finite length'' and, in fact, this consequence should be sufficient for the existence and uniqueness of the nonlinear system $f(x)=y$ whenever $f$ is a local homeomorphism. Supported by this, in 1920 Levy \cite{levy} proved the following for $X=Y$  the space of square integrable functions on $[0,1]$ and $f$  a local homeomorphism of $X$ into $Y$: {\it Suppose that if an open curve $q$ in $X$ is mapped homeomorphically by $f$ on an open segment $p$, must be of finite length; then $f$ is a homeomorphism of $X$ onto $Y$}.  As a consequence he gives an extension of the Hadamard global inversion theorem in this setting. 

In our context, we have the following. Let $f:X\rightarrow Y$ be a local diffeomorphism between Finsler manifolds. Assume $Y$ is connected and $X$ is complete. Because the covering maps have the unique path lifting property, we can conclude that  $f$ is a smooth covering map if and only if:
\begin{enumerate}[label={\bf (C12)}]
\item \label{rectifiablelifting} Every local lifting of a $C^1$ path has finite length.
\end{enumerate}
The Proposition 3.28 of \cite{bridson} is a related result for certain length spaces that are locally uniquely geodesic (which includes the Riemannian manifolds, but not all Banach spaces) such that $\ell(q)\leq\ell(f\circ q)$ for all paths $q$ in $X$. In our context, we can extend this result for Finsler manifolds by means of the condition:
\begin{enumerate}[label={\bf (C13)}]
\item \label{bridson} $X$ is complete and the length of every  $C^1$ path in $X$ is not bigger than the length of its image under $f$.
\end{enumerate}
Clearly condition \ref{bridson} implies condition \ref{rectifiablelifting}. Indeed, if $p$ is a $C^1$ path starting at $f(X)$ and $q$ is a local lifting of $p$ by $f$ then $\ell(q)\leq\ell(p)<\infty$.
\end{remark}

\subsubsection{Coercivity and the Hadamard integral condition}
A map $f:\mathbb{R}^n\rightarrow \mathbb{R}^n$ is proper if and only if it is {\it norm-coercive}; see Theorem 3.3 of \cite{nonlinearspectraltheory}, namely:
$$|f(x)|\rightarrow\infty \mbox{ as } |x|\rightarrow \infty.$$
That is, for any $\varrho\geq 0$ there is $\rho\geq 0$ such that $|f(x)|>\varrho$ if $|x|>\rho$. It is easy to see that $f$ is norm-coercive if and only if the pre-image $f^{-1}(B)$ of any bounded subset $B$ of $Y$ is bounded by $X$. In particular {\it $f:\mathbb{R}^n\rightarrow\mathbb{R}^n$ is a norm-coercive local homeomorphism  if and only if $f$ is a global homeomorphism}. This characterization supports the false but intuitive idea that  a continuous bijection between vector normed spaces ``must'' be coercive. Nevertheless, for a infinite-dimensional Banach space $X$, a homeomorphism $f:X\rightarrow X$ can be constructed such that $f$ maps $X\setminus B$ into $B$ where $B$ is a ball in $X$ \cite{tseng}. So this homeomorphism is not norm-coercive. To complete the picture, the reader is referred to Example 3.12 of  \cite{nonlinearspectraltheory} for an infinite-dimensional  example of a norm-coercive but non-proper map.

For special cases there are some results along this line, for example every   locally injective norm-coercive compact perturbation of the identity is a global homeomorphism \cite{rheinboldt}.  In general, {\it a local diffeomorphism $f:X\rightarrow Y$ between Banach spaces is a global one provided it is  norm-coercive  and $\|df(x)^{-1}\|\leq g(|x|)$ for some continuous positive function $g$ on $\mathbb{R}$.} Note that such a function $g$ exists if and only if $\sup_{|x|\leq \rho}\|df(x)^{-1}\|<\infty$ for all $\rho>0$.  The last statement in italics was proposed by Plastock and proven via condition $L$ \cite{plastock}.  See \cite{zampieri} for an alternative proof.   Another generalization for metric spaces but in terms of  $D_x^-f$ ---which includes  Finsler manifolds--- can be found in \cite{oliviajesuscoverings}. 

 From now and throughout this subsection  we are going to suppose that $X$ and $Y$ are both connected Finsler manifolds endowed  with the Finsler metrics $d_X$ and $d_Y$, respectively. So, a more than justified version in our setting can be established via the injectivity or surjectivity indicator by means of the following condition for some $x_0\in X$:
\begin{enumerate}[label={\bf (P)}]
\item \label{infinitedim} $X$ is complete and for any $\rho>0$ there exists $\alpha_\rho>0$   such that $\mu(x) >\alpha_\rho$ if $d_X(x_0,x)\leq \rho$.
\end{enumerate}
Let $f:X\rightarrow Y$ be a Fredholm map of index zero. Clearly \ref{infinitedim} implies that $f$ a local diffeomorphism. Furthermore, because the mapping $x\mapsto \mu(x)$ is continuous, if $X$ is finite-dimensional then  condition \ref{infinitedim} is equivalent to $f$ being a local diffeomorphism.  Recall, a function $f$ between connected Finsler manifolds is {\it coercive} if $f^{-1}(B)$ is bounded provided $B$ is bounded, that is, for some $y_0\in Y$ and $x_0\in X$, $d_Y(f(x),y_0)\rightarrow\infty$ as $d_X(x,x_0)\rightarrow\infty$. We shall say that a map $f:X\rightarrow Y$ satisfies the {\it Plastock condition} if:
\begin{enumerate}[label={\bf (C14)}]
\item \label{katriel} $f$ is coercive and satisfies \ref{infinitedim} for some $x_0\in X$.
\end{enumerate}
{\it If $f$ satisfies the Plastock condition then it is a smooth covering map}. This can be seen as a direct consequence of the Earle-Eells condition since, for every $y\in Y$ we can consider a (small enough) bounded set $V$ containing $y$ and dominium of a chart centered at the origin.

Note that \ref{infinitedim} holds if and only if  $\inf_{d_X(x_0,x)\leq \rho}\mu(x)>0$ for all $\rho>0$.  Now, let 
$$\varrho(r)=\int_{0}^r \inf_{d_X(x_0,x)\leq \rho}\mu(x) d\rho.$$
The function $\rho\mapsto \inf_{d_X(x_0,x)\leq \rho}\mu(x)$ is nonincreasing, therefore a sufficient (obviously not necessary) condition for \ref{infinitedim} is: 
\begin{enumerate}[label={\bf (C15)}]
\item \label{hadamardintegral} $X$ is complete and $\lim_{r\rightarrow\infty}\varrho(r)=\infty$.
\end{enumerate}
And this limit is nothing but the {\it  infinite-dimensional version of the Hadamard integral condition}. As may be expected and indeed pointed out below, for the index zero case condition \ref{hadamardintegral} implies that $f$ is a smooth covering map and $Y$ is complete. This is  a consequence of the following remarkable fact. Let $r>0$ be fixed and $\varrho=\varrho(r)$. If $f$ is a local diffeomorphism, since every rectifiable path $p$ in $B_\varrho(f(x_0))$ with $p(0)=f(x_0)$ and $\ell(p)<\varrho$ can be lifted to a path in $B_r(x_0)$ then: \begin{equation}\label{janjohansson}\varrho>0\mbox{ implies }B_\varrho(f(x_0))\subset f(B_r(x_0)).\end{equation} 
This has been noted by John \cite{john} for Banach spaces and generalized  for length spaces in terms of $D_x^-f$ in \cite{oliviajesusisabel}.  A direct proof for Finsler manifolds can be done in terms of $\mu(x)$ using the same arguments; see proof of Lemma \ref{hadamardimpliesplastock} in the appendix. Let $x_0$ be a solution solution of $f(u)=y_0$ and suppose that $\varrho>0$ for some $r>0$. Note that \eqref{janjohansson} implies that if $d_Y(y,y_0)<\varrho$ then there is a solution $x$ of $f(u)=y$ such that $d_X(x,x_0)<r$.

Even more, if the mapping $f:X\rightarrow Y$ satisfies \ref{hadamardintegral} and $Y$ is simply connected then $f$ is a global coercive diffeomorphism, hence $f$ satisfies the condition \ref{katriel}. Indeed, let $B$ be a bounded set in $Y$. There is $R>0$ such that $B\subset B_{R}(f(x_0))$. Since $\lim_{r\rightarrow\infty}\varrho(r)=\infty$  there is $s>0$ such that $R=\varrho(s)$. So, $B\subset B_R(f(x_0))\subset f(B_s(x_0))$ and therefore $f^{-1}(B)\subset B_s(x_0)$. Summing up, if $f:X\rightarrow Y$ is a Fredholm map of index zero between  connected Finsler manifolds satisfying {\rm \ref{hadamardintegral}} then:
\begin{enumerate}
\item[\buletita] $f$ is a smooth covering map.
\item[\buletita] $Y$ is  complete.
\item[\buletita] If $Y$ is simply connected then $f$ satisfies the Plastock condition, in particular $f$ is a coercive global diffeomorphism. 
\end{enumerate}

\begin{remark}{\it Ray-coercive maps.} 
 A map $f:X\rightarrow Y$ between Banach spaces is said to be {\it ray-coercive} at $y_0$ if the pre-image of  the line joining $y_0$ and $z$ is bounded for any $z\in Y$ and some $y_0\in Y$.  A coercive map is ray-coercive, but not vice versa; see Example 3.11 of \cite{nonlinearspectraltheory}. Furthermore, every ray-proper function is ray-coercive and the converse is true if $X$ is finite-dimensional, see e.g. Example 3.14 of \cite{nonlinearspectraltheory}. It is easy  to conclude that a Fredholm map of index zero $f$ is a global diffeomorphism provided:
\begin{enumerate}[label={\bf (C16)}]
\item \label{weakkatriel}  $f$ is ray-coercive at some $y_0\in f(X)$ and satisfies \ref{infinitedim}.
\end{enumerate}
Suppose that $f$ is a Fredholm map of index $0$  such that \ref{infinitedim} holds for $x_0=0$, that is, for any $\rho>0$ there exists $\alpha_\rho>0$   such that $\mu(x) >\alpha_\rho \mbox{ if } |x|\leq \rho$. Thus $f$ is a local diffeomorphism and if $y_0\in f(X)$ then  {\it $f$ is ray-coercive at $y_0$ if and only if $f$ is ray-proper at $y_0$}. Actually, if $f$ is ray-coercive at $y_0\in f(X)$ and satisfies \ref{infinitedim} then any local lifting $q$ of a ray starting at $y_0$  is contained in a ball $B_\rho(0)$ for some $\rho>0$ hence $q$ has finite length. So $f$ is a global diffeomorphism. Then $f$ is a proper map, thus it is a ray-proper map at $y_0$.
\end{remark}

\begin{remark}
An interesting fact is that for a linear map $T:X\rightarrow Y$ the properness and norm-coercivity of $T$ are each equivalent to the existence and boundedness of $T^{-1}$ on $\Range T$ \cite[p. 67, p. 73]{nonlinearspectraltheory}. As a consequence, for linear Fredholm maps of index $0$ (since $dT(x)=T$ for all $x\in X$) the following characterization of a linear isomorphism can be concluded.  Let $T:X\rightarrow Y$ be a {\it linear} Fredholm map of index $0$, then the following statements are equivalent:
\begin{enumerate}
\item[\buletita] $T$ is a proper map.
\item[\buletita] $T$ is norm-coercive.
\item[\buletita] $T$ satisfies the Hadamard integral condition.
\item[\buletita] $T$ is a linear isomorphism.
\end{enumerate}
\end{remark}

As already said in the introduction,  Katriel \cite{katriel} established an alternative technique to the monodromy argument for global homeomorphism theorems for maps between metric spaces. The principal idea is to show that mountain-pass theorems can be used to prove new global inversion results as well as new proofs of known theorems. Theorem 6.1 of \cite{katriel} asserts that:  A local homeomorphism  $f:X\rightarrow Y$ is a  global homeomorphism provided that for all $\varrho>0$ and for some $y_0\in Y$:  $$\inf\{\mbox{\rm sur}(f,x): d(f(x),y_0)<\varrho\}>0,$$  where $X$ and $Y$ are complete path-connected metric spaces such that $X$ remains path-connected after the removal of any discrete set and $Y$ is a ``nice'' space, that is, for each $y\in Y$ there is a continuous functional $g_y:Y\rightarrow\mathbb{R}$  satisfying a PS-condition (non-smooth version) and possessing a unique minimizer and a discrete set of maximizers as the only critical points. Here,  $\mbox{sur}(f,x)$ is the  {\it surjection constant} of $f$ at $x$, originally introduced by Ioffe \cite{ioffe} in order to get a generalization of Plastock's results  (also via condition $L$) for non differentiable maps between Banach spaces, namely:
$$\mbox{\rm sur}(f,x)=\liminf_{r\rightarrow 0}\frac{1}{r}\sup\{R\geq 0:B_R(f(x))\subset f(B_r(x))\}.$$
Fortunately,  for a local diffeomorphism $f:X\rightarrow Y$ between  connected and complete Finsler manifolds we have: $$D_x^-f=\mbox{sur}(f,x)=\mu(x).$$ See Remark 3.4 and Example 3.2 of  \cite{oliviajesuscoverings} and Proposition 3.11 of  \cite{jesusyoscar}. Therefore, in our setting, the {\it Katriel condition} can be established in terms of the injectivity or surjectivity indicator and  the Finsler distance, that is:

\begin{enumerate}[label={\bf (C17)}]
\item \label{katrielmain} $X$ is complete and $\inf\{\mu(x): d_Y(f(x),y_0)<\varrho\}>0$ for all $\varrho>0$ for some (then for all) $y_0\in Y$.
\end{enumerate}

Note that if $f$ is a Fredholm map of index zero then condition \ref{katrielmain} implies that $f$ is a local diffeomorphism. Furtheremore, we can  prove that  condition \ref{katrielmain} implies the continuation property for all $C^1$ paths: Let $p$ be a $C^1$ path in $Y$ beginning at $y_0$ and let $q$ be a local lifting of $p$. Since the image of $p$ is compact, there is $\varrho>0$ such that $d_Y(p(t),y_0)<\varrho$ for all $t\in I$. In particular, $d_Y(f(x),y_0)<\varrho$ for all $x$ in the image of $q$. Therefore $\ell(q)<\infty$. In other words, {\it if $X$ and $Y$ are connected Finsler manifolds and $f:X\rightarrow Y$ is a Fredholm map of index zero then condition {\rm \ref{katrielmain}} implies that $f$ is a smooth covering map.} 

In comparison with the Katriel approach, it is worth mentioning that it is not clear which Finsler manifolds $Y$ are nice spaces. Actually, Katriel gives only two examples: Banach spaces with $g_y(z)=|z-y|$ and infinite-dimensional Hilbert manifolds with $g_y(z)=1-e^{\|\pi(z)-y\|_y}$ for $z\neq -y$ and $g_y(-y)=1$, where the map $\pi:Y\setminus\{-y\}\rightarrow T_yY$ is the stereographic projection. Cartan-Hadamard Finsler manifolds can be added to the list.  Nevertheless, all these examples are simply connected spaces, so at the moment an example hasn't been presented to test the effectiveness of the monodromy argument. We shall return to this point in the next subsection.

As Katriel notes: {\it the Plastock condition implies the Katriel condition}. Indeed, since $f$ is coercive, for every  $\varrho>0$ there is $\rho>0$ such that $d_Y(f(x),y)\geq\varrho$ if $d_X(x,x_0)\geq \rho$ for some $x_0$. Therefore, $$\inf\{\mu(x): d_Y(f(x),y_0)<\varrho\}\geq\inf\{\mu(x): d_X(x,x_0)<\rho\}\geq\alpha_\rho>0.$$
See Theorem 6.2 of \cite{katriel}. In short, if $f:X\rightarrow Y$ is a local diffeomorphism between connected and complete Finsler manifolds and $Y$ is simply connected then each of the following statement implies the next:
\begin{enumerate}
\item[\buletita] $f$ is an expansive map.
\item[\buletita] $f$ satisfies the Hadamard-Levy condition.
\item[\buletita] $f$ satisfies the Hadamard integral condition.
\item[\buletita] $f$ satisfies the Plastock condition, in particular $f$ is a coercive map.
\item[\buletita] $f$ satisfies the Katriel condition.
\item[\buletita] $f$ satisfies the Earle-Eells condition.
\end{enumerate}

\subsubsection{The special case of uniform lower bound}
The classical theorem in Riemannian geometry
proved by  Ambrose \cite{ambrose}  asserts the following.  Let  $(X,h)$ and $(Y,g)$ be  Riemannian manifolds of dimension $n$. If $f:X\rightarrow Y$ is a  surjective $C^{\infty}$  map such that:
\begin{enumerate}[label={\bf (C18)}]
\item \label{localisometry} $X$ is complete and $f$ is a {\it Riemannian local isometry} i.e. for every $x\in X$, $v,w\in T_xX$ and $y=f(x)$ it holds that $g_{y}\left(df(x)v,df(x)w\right)=h_x(v,w),$
\end{enumerate}
then $f$ is a smooth covering map, hence $f$ is a {\it Riemannian covering}.  If $Y$ is simply connected then $f$ is indeed a {\it Riemannian isometry}. A simple calculation shows that condition \ref{localisometry} implies that $\inj df(x)=1$ for all $x\in X$.  Condition \ref{localisometry}  makes  sense also for infinite-dimensional Riemannian manifolds. Note that for an infinite-dimensional version, we can substitute the hypothesis $\dim X=\dim Y=n$ by requesting that $f$ be a Fredholm map of index zero (the surjectivity condition on $f$ may be changed by the connectedness of $Y$). Finally, by the chain rule, if $f$ is a Riemannian local isometry then it satisfies  \ref{bridson}.

Now, suppose that $f:(X,h)\rightarrow (Y,g)$ is a local diffeomorphism between possibly infinite-dimensional Riemannian manifolds. Assume $Y$ is connected and $X$ is complete. A typical step in the proof of the Cartan-Hadamard theorem ---e.g. Theorem 6.9 of \cite{lang}--- uses the fact that, if a constant $\alpha>0$ exists such that for all $x\in X$ and $v\in T_xX$ such that $\|df(x)v\|_{g}\geq\alpha\|v\|_h$
then $f$ is a covering map. In the above inequalities: $$\|v\|_{h}^2=h_x(v,v)\mbox{ and }\|df(x)v\|_{g}^2=g_y(df(x)v,df(x)v).$$ 
In our notation, this  means that  {\it if   $\inj df(x)>\alpha$ for all $x\in X$ then $f$ is a covering map}. The Cartan-Hadamard Theorem follows for the particular case $$f=\exp_p:T_pM\rightarrow M$$ where $M$ is a geodesically complete manifold of semi-negative curvature since, in this case $\alpha=1$. The usual proof of the last statement in italics consists of reducing the demonstration to the case where $f$ is a local isometry and then Ambrose Theorem is used. Relatively recently, a similar result with an analogous approach was given by Neeb  \cite{neeb} for Finsler manifolds $(X,\|\cdot\|_X)$ and $(Y,\|\cdot\|_Y)$ ($Y$ is a manifold with spray)  through the inequality:
$$\|Tfv\|_{Y}\geq\alpha\|v\|_X \mbox{ for all } v\in TX$$ where $Tf:TX\rightarrow TY$ is the tangent map defined to be $df(x)$ on each fibre of $T_xX$. Of course, Neeb's condition implies that $\inj df(x)>\alpha$ for all $x\in X$. Therefore  Neeb's result can be deduced from the above arguments using the Hadamard-Levy condition in terms of the injectivity indicator. In this case:
\begin{itemize} 
\item[\buletita]  If  $X$ is  connected then the Hadamard integral condition holds trivially,  regardless of $x_0$. In particular, $Y$ is complete and:  $$B_{\alpha r}(f(x))\subset f(B_r(x)) \mbox{ for all }x\in X \mbox{ and }r>0.$$
\item[\buletita] If $Y$ is simply connected then  $f$ is a coercive global diffeomorphism.
\end{itemize}

\subsection{Palais-Smale conditions}
In the middle of the sixties, Palais and Smale \cite{palaismale} introduced the {\it condition} (C) for functionals $F:X\rightarrow \mathbb{R}$ where $X$ is a Riemannian manifold,  possibly infinite-dimensional. Suppose for convenience that $X$ is a Hilbert space. A functional $F$ satisfies the condition (C) if the closure of any nonempty subset $S$ of $X$ on which $f$ is bounded but on which $|\nabla F|$ is not bounded away from zero, contains a critical point of $F$. As usual  $\nabla F(x)$ is the gradient of $F$ at $x$ defined in terms of the Fr\'echet differential by 
$dF(x)w=\langle\nabla F(x),w\rangle$.  The origin  of the condition (C) is the study of the asymptotic properties of the {\it gradient flow} of a $C^2$ functional, namely the solutions $x(t)$ of the Cauchy problem \cite{mawhin2}:
$$\dot x=-\nabla F(x), \hspace{1cm}x(0)=x_0,$$
The function $t\mapsto F(x(t))$ is nonincreasing. If $x(t)$ is defined for all positive $t$ and $\lim_{t\rightarrow+\infty} F(x(t))$ is finite then there is a sequence $\{t_n\}$ such that $\nabla F(x(t_n))\rightarrow 0$. As is  also point out in \cite{mawhin2}: ``the question is then to find conditions upon $F$ under which this sequence of {\it almost critical points} of $F$ provides a real one.'' 

The first work that the author found in the literature that established a relationship between condition (C) and global inversion theorems corresponds to Gordon \cite{gordon}. The paper consists basically of an alternative proof of the fact that a local diffeomorphism $f:\mathbb{R}^n\rightarrow\mathbb{R}^n$ is a global one if and only if it is a proper map.  The idea to show that $f$ is onto provided it is proper is the following. Let  $$F_y(x)=\frac{1}{2}|f(x)-y|^2.$$ Since $F_y$  is also proper, it  satisfies the condition (C) and the gradient flow exists for all positive time. So a critical point $x^*$ of $F_y$ exists as the limit of a sequence $\{x(t_n)\}$ in the gradient flow of $F_y$ such that $f(x^*)=y$. Gordon also used the gradient flow of $F_y$ to show that $f$ is one-to-one. An alternative and very easy proof of the injectivity of $f$ can be found in the introduction to \cite{katriel} by means of the simplest mountain-pass theorem in $\mathbb{R}^n$ applying to $F_y$.

Note that in the generalizations of the mountain-pass theorems in non locally compact settings it is usual to replace the  norm-coercivity assumption  by a generalization or variant of the Palais-Smale type condition. So, a suitable Palais-Smale type condition for $F_y$ must  guarantee the bijection of $f$. This reflection was noted by Rabier \cite{rabier2} independently of Gordon and Katriel a year before the latter.

Recall, a $C^1$  functional $F$ satisfies the PS-{\it condition} if any sequence $\{x_n\}$ in $X$ such that $F(x_n)$ is bounded and $\|\nabla F(x_n)\|\rightarrow$ 0 ---called PS-{\it sequence}--- contains a convergent subsequence, whose limit is then a critical point of $F$. This is a stronger condition of the condition (C); see Section 3 of  \cite{mawhin2}, but has been widely used in the context of Banach spaces, as well as in the following localized form: a functional $F$ satisfies the $\mbox{PS}_c$-{\it condition} if any sequence $\{x_n\}$ in $X$ such that $F(x_n)\rightarrow c$ and $\|\nabla F(x_n)\|\rightarrow$ 0 ---called $\mbox{PS}_c$-{\it sequence}--- contains a convergent subsequence. 

Recently, Idczak {\it et al.} \cite{idzak1} adapted the ideas of Katriel \cite{katriel}  to get a global inversion theorem for a $C^1$ map $f:X\rightarrow Y$  between Hilbert spaces by means of the functional: $$F_y(x)=\frac{1}{2}|f(x)-y|^2.$$ 
 as ``an alternative'' to the Plastock condition. Specifically, they proved that a local diffeomorphism $f$ is a global one, provided:
\begin{enumerate}[label={\bf (C19)}]
\item \label{palaismale}  $F_y$  satisfies the  $\mbox{PS}$-condition for all $y\in Y$.
\end{enumerate} 
See also \cite{galewski} for some recent related results along this line. The connection between Idczak's result and the above conditions can be easily made by means of the following fact:  {\it If $f:X\rightarrow Y$ is a Fredholm map of index $0$ between Hilbert spaces then  the Katriel condition implies  that $F_y$  satisfies the  $\mbox{PS}$-condition for all $y\in Y$}; see Lemma \ref{lemaconectionkatrielps} in the appendix. On other hand, the functional $F_0=\frac{1}{2}|f|^2$  is bounded below. By the Ekeland Variational Principle,  if $F_0$ satisfies  the $\mbox{PS}$-condition then it is a coercive map. Therefore, so is $f$.  In summary, if  $f:X\rightarrow Y$ is a Fredholm map of index $0$ between Hilbert spaces and satisfies {\rm\ref{infinitedim}} then the following statements are equivalent:
\begin{enumerate}
\item[\buletita] $F_y$  satisfies the  $\mbox{PS}$-condition for all $y\in Y$.
\item[\buletita] $f$ is a coercive global diffeomorphism.
\item[\buletita] $f$ satisfies the Plastock condition.
\item[\buletita] $f$ satisfies the Katriel condition.
\end{enumerate}
The author believes that condition \ref{palaismale} and Lemma \ref{lemaconectionkatrielps} may be carried out to the Banach spaces setting using the Clark subgradient for the functions $F_y$ with a suitable definition of the Palais-Smale sequence. Perhaps it can also be extended to the Cartan-Hadamard Finsler manifolds. The appropiate adequation  for Finsler manifolds is not clear in terms of the Finsler distance, since the critical points of the distance function are  involved.

The proof of Lemma \ref{lemaconectionkatrielps} basically contains two ideas: the Katriel condition implies that there are no $\mbox{PS}_c$-sequences for $F_y$ with $c\neq 0$ and  every $\mbox{PS}_0$-sequence for $F_y$ converges trivially to the minimum of $F_y$. Along this line, Rabier gives a characterization of the global diffeomorphism between $C^1$ Finsler manifolds in terms of sort of ``generalized $\mbox{PS}$-sequences''; see Theorem 5.3 of \cite{rabier}. He considers for a map $f:X\rightarrow Y$ between Finsler manifolds the following condition:
\begin{enumerate}[label={\bf (C20)}]
\item \label{strongsubmersion} $X$ is complete and $f$ is a {\it strong submersion}, that is: there is no sequence $\{x_n\}$ from $X$ with $f(x_n)\rightarrow y\in Y$ and $\sur f(x_n)\rightarrow 0$.
\end{enumerate} 
Rabier establishes that {\it if $f$ is a local diffeomorphism, $Y$ is simply connected, and $X$ is complete then $f$ is a strong submersion if and only if it is a global diffeomorphism}, arguing without more details, that it is enough an analogous version for Finsler manifolds of the Plastock's global inversion theorem via condition $L$ \cite[Rmk. 4.2]{rabier}. In the proof of Theorem \ref{firstcharacterization} and Lemma \ref{lemmafinitefibre} sufficient arguments have already been given to justify this statement, since if $X$ is complete,  the definition of a  strong submersion is a simply rephrasing of the condition that $f$ satisfies the Earle-Eells condition. In other words, condition \ref{earleell} is equivalent to condition \ref{strongsubmersion}.

If $X$ and $Y$ are Hilbert spaces and $f$ is a local diffeomorphism then condition {\rm\ref{palaismale}} implies that $f$ is a  strong submersion, hence $f$ lifts lines. It is important to note that a direct proof  of this fact  hasn't been given (and at the moment the author doesn't know how). Instead, using arguments of critical point theory, it has been showed that  condition \ref{palaismale} implies that $f$ is a global diffeomorphism. But this is not an exception, for example Xavier \& Nollet \cite{xaviernollet} proved that if $f:\mathbb{R}^n\rightarrow\mathbb{R}^n$ is a local homeomorphism such that:
\begin{enumerate}[label={\bf (C21)}]
\item \label{xavier} $f_v(x)=\langle f(x),v\rangle$ satisfies the Palais-Smale conditon for all nonzero $v\in\mathbb{R}^n$.
\end{enumerate} 
 then it is bijective. Again, the monodromy argument doesn't seems to be natural in this case. The proof of Xavier and Nollet is based on arguments involving degree theory and cannot be extended in a general form to the infinite-dimensional setting, only for restricted classes of maps. Note that if for all $v\neq 0$, $$\inf_{x\in X}|df(x)^*v^*|>\alpha_v>0$$ then $f_v$ satisfies trivially the PS-condition since $|\nabla f_v(x)|=|df(x)^*v^*|$.

Also, with this technique Xavier and Nollet proved a significantly simpler version of the Hadamard theorem by means of integral conditions with parameter $v$: $$\int_0^\infty \min_{|x|=\rho}|\nabla f_v(x)| d\rho=\infty, \mbox{ for all }v\neq 0.$$
Some recent extensions of this kind of theorem for finite-dimensional manifolds can be found in \cite{xaviernegativecurvature}. So, a pertinent question is whether we can replace all of the above metric conditions in terms of $\mu(x)$ by a family of metric conditions with parameter $v\neq 0$ in terms of $|df(x)^*v^*|$ in the finite-dimensional case.  

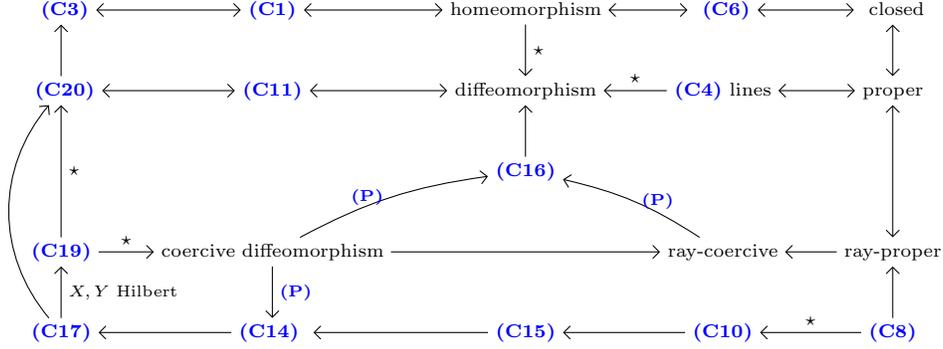
\begin{figure}[htbp]
\begin{center}
{\scriptsize
\begin{tikzpicture}[>=angle 90]
\matrix(b)[matrix of math nodes,
row sep=2.5em, column sep=2.5em,
text height=1.5ex, text depth=0.25ex]
{\mbox{ \ref{browder}}& \mbox{\ref{pathlifing}}&\mbox{homeomorphism} & \mbox{ \ref{weaklyproper}} & \mbox{ closed}\\
\mbox{ \ref{strongsubmersion}}& \mbox{\ref{earleell}}&\mbox{diffeomorphism} & \mbox{\ref{rheinboldtc} lines} & \mbox{proper}\\
& &\mbox{\ref{weakkatriel}} &  & \\
\mbox{\ref{palaismale}}& \mbox{coercive diffeomorphism}&& \mbox{ray-coercive} & \mbox{ray-proper}\\
\mbox{\ref{katrielmain}}& \mbox{\ref{katriel} }& \mbox{\ref{hadamardintegral}} & \mbox{\ref{hadamard-levy}} & \mbox{\ref{expansive}}\\
};

\path[<->]
(b-1-1) edge (b-1-2);
\path[->]
(b-2-1) edge (b-1-1);
\path[<->]
(b-1-2) edge (b-1-3);
\path[<->]
(b-1-3) edge (b-1-4);
\path[<->]
(b-1-4) edge (b-1-5);
\path[<->]
(b-1-5) edge (b-2-5);
\path[<->]
(b-2-1) edge (b-2-2);
\path[<->]
(b-2-2) edge (b-2-3);
\path[<-]
(b-2-3) edge node[above]{$\star$} (b-2-4);
\path[<->]
(b-2-4) edge (b-2-5);
\path[->]
(b-1-3) edge node[right]{$\star$} (b-2-3);
\path[->]
(b-3-3) edge (b-2-3);
\path[->,font=\tiny]
(b-4-2) edge [bend left=10]  node[left]{\mbox{\ref{infinitedim}}} (b-3-3);
\path[->]
(b-4-2) edge (b-4-4);
\path[->,font=\tiny]
(b-4-4) edge [bend right=10]  node[right]{\mbox{\ref{infinitedim}}} (b-3-3);
\path[->]
(b-4-1) edge node[right]{$\star$} (b-2-1);
\path[->]
(b-4-1) edge node[above]{$\star$} (b-4-2);
\path[<->]
(b-4-5) edge (b-2-5);
\path[->]
(b-4-5) edge (b-4-4);
\path[<-]
(b-5-1) edge (b-5-2);
\path[<-]
(b-5-2) edge (b-5-3);
\path[<-]
(b-5-3) edge (b-5-4);
\path[<-]
(b-5-4) edge  node[above]{$\star$} (b-5-5);
\path[->]
(b-5-5) edge (b-4-5);
\path[->,font=\tiny]
(b-5-1) edge  node[right]{\mbox{$X, Y$ Hilbert}} (b-4-1);
\path[->,font=\tiny]
(b-5-1) edge [bend left=40]  (b-2-1);
\path[<-,font=\tiny]
(b-5-2) edge  node[right]{\ref{infinitedim}} (b-4-2);
\end{tikzpicture}
}
\end{center}
\caption{Implications and equivalences for  a locally injective Fredholm maps of index $0$. The $\star$ in the arc means that $f$ must have no critical points. }
\label{relationshipbanach}
\end{figure}

\subsection{Weighted conditions}
In the late sixties, Meyer extended condition \ref{hadamard-levy} where $\|df(x)^{-1}\|$ is allowed to go to infinity at most linearly in $\|x\|$; see Theorem 1.1 of \cite{meyer}. More precisely, he proved that  a  local diffeomorphism $f:\mathbb{R}^n\rightarrow\mathbb{R}^n$ which has a locally Lipschitz continuous Fr\'echet derivative  such that  $\|df(x)^{-1}\|\leq a \|x\| +b$ for all $x\in\mathbb{R}^n$ for some $a,b>0$  is a global homeomorphism.  Note that the Meyer result can be deduced from the original Hadamard theorem. His criterion can be substituted by $(a \rho +b)^{-1}\leq \inf_{|x|=\rho}\|df(x)^{-1}\|^{-1}$ for all $\rho\geq 0$. Since $\int_0^\infty \frac{1}{a \rho +b}d\rho=\infty$ therefore $f$ satisfies the {\it finite} dimensional version of the Hadamard integral condition ---the infimum on the spheres instead of the balls--- so it is a global diffeomorphism.

For a function $f:X\rightarrow Y$ between Banach spaces, Rheinboldt  gives an extension of Meyer's result for compact perturbations of the identity, see Theorem 3.12 of \cite{rheinboldt}. Furthermore, in the late eighties Ioffe  actually proved a very general statement in terms of the surjection constant of a continuous locally one-to-one function $f$ and a  lower semicontinuous function $\eta:[0,\infty)\rightarrow (0,\infty)$: If $\eta$ is a {\it weight}, namely:  $$\int_{0}^\infty \eta(\rho) d\rho=\infty,$$  and $\mbox{sur}(f,x)\geq \eta(|x|)$ for all $x\in X$ then 
$f$ is a homeomorphism onto $Y$;  see Theorem 2 of \cite{ioffe}. As we have already pointed out, if $f$ is a local diffeomorphism $\mbox{sur}(f,x)=\|df(x)^{-1}\|^{-1}=\mu(x)$. As before, the Ioffe condition can be rewritten as:
$$\eta(\rho)\leq\inf_{\|x\|=\rho}\mu(x), \mbox{ for all }\rho\geq 0.$$
 This makes a connection with the relatively recent work of Li {\it et al.} \cite{hongxu} who rediscovered this fact when $\eta$ is  a continuous or  a nonincreasing weight. A definite example is, of course, when $f$  satisfies  condition \ref{hadamardintegral} since in this case  $\eta(\rho)$ is the nonincresing weight $\rho\mapsto\inf_{\|x\|\leq \rho}\mu(x)$. The Ioffe formulation can be established for mappings between connected Finsler manifolds, replacing $|x|$ by $d_X(x_0,x)$ for some arbitrary $x_0\in X$. We shall say that a map between connected Finsler manifolds satisfies the {\it Ioffe condition} if:
\begin{enumerate}[label={\bf (C22)}]
\item \label{ioffe} $X$ is complete and there is a {\it continuous} weight $\eta:[0,\infty)\rightarrow (0,\infty)$ such that $\mu(x)\geq\eta(d_X(x_0,x))$ for all $x\in X$ and some $x_0\in X$.
\end{enumerate}
Actually, for mappings between connected Finsler manifolds, Rabier \cite[Rmk. 4.4]{rabier} considers the weighted version of a strong submersion. So, we set the condition:
\begin{enumerate}[label={\bf (C23)}]
\item \label{strongsubmersionweighted} $X$ is complete and $f$ is a {\it strong submersion with continuous weight}, that is: there exists a continuous weight $\eta(\rho)=\frac{1}{\omega(\rho)}$ such that there is no sequence $\{x_n\}$ in $X$ with $f(x_n)\rightarrow y\in Y$ and $$\sur f(x_n)\omega(d_X(x_0,x_n))\rightarrow 0$$ for some  $x_0\in X$.
\end{enumerate}
In fact, in his previous work  \cite{rabier3} Rabier considers the weight $\omega(\rho)=1+\rho$ for functions between Banach spaces, motivated by Cerami's generalization of the Palais-Smale condition \cite{cerami}. See also Section 8 in \cite{mawhin}. He asserts that the Gr\"onwall's Lemma is a  way to check that condition \ref{strongsubmersionweighted} is sufficient for a local diffeomorphism to be a covering map; see Remark 4.4 and Theorem 5.3 of \cite{rabier}. 
 
\begin{remark}{\it Change of metric}.
 An illustrative proof of the fact that condition \ref{strongsubmersionweighted} carries over to a covering map is the following.  In order to simplify the exposition, assume that $(X, |\cdot|)$ is a Banach space. Let $\eta(\rho)=\frac{1}{\omega(\rho)}$ be a continuous weight. Consider the following weighted length of a path $\tilde\ell(\alpha)=\int_{0}^1\eta(|\alpha(t)|)|\dot \alpha(t)|dt$ and set:
  $$\tilde d(x,x')=\inf\{\tilde\ell(\alpha):\alpha \mbox{ is a } C^1 \mbox{ path connecting } x \mbox{ with } x'\}.$$
 According to Theorem 4.1 of \cite{corvellec} (see also Section 3 in \cite{plastock}) $\tilde d$ is a metric  such that $(X,\tilde d)$ is complete if and only if $X$ is complete with the distance associated to the given norm. If $f:X\rightarrow Y$ is a local diffeomorphism and $q$ is a local lifting of a rectifiable path $p$ in $Y$  then by the chain rule: 
\begin{equation}\label{meanvalueweighted}\tilde\ell(q)\cdot\inf\{\mu(x) \omega(|x|): x\in \mbox{image of }q\}\leq \ell(p).\end{equation}
Compare with \eqref{meanvalue}. So, $\tilde\ell(q)<\infty$ provided  the above infimum is positive. Since $(X,\tilde d)$ is also complete, the path $q$ can be extended to whole interval $I=[0,1]$. Therefore, we can carry on as in the proof of Theorem \ref{firstcharacterization} to conclude that the corresponding weighted version of the Earle-Eells condition, equivalent to \ref{strongsubmersionweighted}, implies that $f$ is a smooth covering map.
\end{remark}

Furthermore, we can proceed stepwise as in the proof of  Lemma 4.4 of \cite{oliviajesuscoverings} to conclude that if $f$ is a Fredholm  map of index $0$ between connected Finsler manifolds then the infimum of $\mu(x)$ over the image of a local lifting $q$ is positive if and only
the infimum of $\mu(x)\omega(d(x_0,x))$ over the same set is positive for some $x_0\in X$ and nonincreasing weight $\eta(\rho)=\frac{1}{\omega(\rho)}$. Therefore, the {\it Earle-Eells condition with nonincreasing weight}:
\begin{enumerate}[label={\bf (C24)}]
\item \label{earleellweighted} $X$ is complete and there exists a  nonincreasing weight $\eta(\rho)=\frac{1}{\omega(\rho)}$ such that for every $y\in Y$ there exists $\alpha>0$ and a neighborhood $V$ of $y$ such that $\mu(x)\omega(d(x_0,x))\geq\alpha, \mbox{ for all } x\in f^{-1}(V)$ and for some $x_0\in X$.
\end{enumerate}
implies that $f$ is a smooth covering map.  We can  deduce that  {\it $f$  satisfies the infinite version of Hadamard integral condition  if and only if there exists a nonincresing weight $\eta$ such that $\mu(x)\geq\eta(d(x_0,x))$ for every $x\in X$ and some $x_0\in X$}.

Note that, if $f$ is Fredholm map of index $0$ that satisfies \ref{strongsubmersionweighted} or \ref{earleellweighted} then the map $f$ is actually a local diffeomorphism. Furthermore, the constant map $\eta(\rho)= 1$ is both continuous and nonincreasing, so a consequence of Lemma \ref{lemmafinitefibre} in the appendix is the following fact: Let $X$ and $Y$ be connected Finsler manifolds. Assume that $X$ is complete and $Y$ is simply connected. If $f:X\rightarrow Y$ is a Fredholm map of index $0$ then the following statements are equivalent:
\begin{enumerate}
\item[\buletita] $f$ is a strong submersion.
\item[\buletita] $f$ is a strong submersion with  continuous weight.
\item[\buletita] $f$ satisfies the Earle-Eells condition.
\item[\buletita] $f$ satisfies the Earle-Eells condition with nonincreasing weight.
\item[\buletita] $f$ is a global diffeomorphism.
\end{enumerate}
This equivalence is no longer true if the space $Y$ is not simply connected; see Example 4.5 of \cite{oliviajesusfibrations}.

\section{Submersions as global projections}
Let $T:\mathbb{R}^n\rightarrow\mathbb{R}^m$  be a linear map, where $n\geq  m$. Consider the linear system $T(x)=y$ with rank of $T$ equal to $m$. An elementary linear algebra argument shows that there is a change of basis $\Phi$ such that $T\Phi^{-1}=P$ where $P:\mathbb{R}^n\rightarrow\mathbb{R}^m$ is a projection map. The nonlinear version of this fact  for continuously differentiable  maps $f:\mathbb{R}^n\rightarrow\mathbb{R}^m$ such that $df(x)$ has rank $m$ for some $x\in\mathbb{R}^n$ is a consequence of the classical Inverse Mapping Theorem \cite[p. 43]{spivak}, in this case $f$ looks in a neighborhood of $x$ like a projection onto $\mathbb{R}^m$. To be more precise and broader,  given a product of two open sets of Banach spaces $W_1\times W_2$ and a Banach space $F$, a mapping $h:W_1\times W_2\rightarrow F$ is said to be equivalent to a projection  if $h$ can be factored into an ordinary projection and a homeomorphism of $W_1$ onto an open subset of $F$. If $f:E\rightarrow F$ is a $C^1$ map such that $df(x)$ is onto and $\Kernel df(x)$ splits for some $x\in X$ ---e.g. if $f:X\rightarrow Y$ is a nonlinear Fredholm map of positive index--- then there exists an open subset $W$ of $x$ and a homeomorphism $\Phi: W\rightarrow W_1\times W_2$ such that the composite map $f\circ\Phi^{-1}$ is equivalent to a projection \cite[p. 19]{lang}. More generally, let $f:X\rightarrow Y$ a $C^1$ map between $C^1$ Banach manifolds. A map $f$ is said to be a  {\it split submersion} if $df(x)$ is onto  and $\Kernel df(x)$ splits for all $x\in X$. If $f$ is a split submersion then  for every $x\in X$ there is a chart $(W,\varphi)$ at $x$, a chart $(V,\psi)$ at $f(x)$, and a homeomorphism $\Phi:\phi(W)\rightarrow W_1\times W_2$ with $W_1$ and $W_2$ open in some Banach spaces such that the map $\psi f\varphi^{-1}\circ \Phi^{-1}$ is equivalent to a projection  \cite[p. 27]{lang}. In particular, if $f$ is surjective then  each fibre $f^{-1}(y)$ is a closed differentiable submanifold of $X$ ($f$ foliates $X$).

\begin{remark}
A Fredholm map with positive index without critical points is, of course, a split submersion. Therefore, if $X$ is connected and $f$ is onto then every $y\in Y$ is a regular value and $f^{-1}(y)$ is a closed submanifold of $X$ with $$\dim f^{-1}(y)=\dim\Kernel df(x)=\Index f.$$ Nevertheless, if $X$ is  Riemannian then any map $f:X\rightarrow Y$ without critical points is a split submersion even if it is not Fredholm. 
\end{remark}

In the same spirit as in previous sections we request global properties of $f$. A particular case of split submersion is obtained when $f$ is a local diffeomorphism which, in a desired situation, is a covering map. A covering space  is generalized by the concept of {\it fibre bundle}. A map $f$ is a fibre bundle if it is onto and there exists a covering $\{V\}$ of $Y$ and a topological space $\mathcal F$ such that each $f^{-1}(V)$ is homeomorphic to $V \times \mathcal{F}$ by a map $\Phi_V$ and $f\circ \Phi_V^{-1}$ is the projection on the first factor. In particular, every fibre $f^{-1}(y)$ is homeomorphic to  $\mathcal F$.  Furthermore, if $Y$ is contractible then $f$ is a trivial fibre bundle with trivialization $\mathcal F\times Y$ homeomorphic to $X$ \cite[p. 102]{spanier}, so we can talk about a ``global projection''.

The problem is: When is a split submersion a fibre bundle? This is an old question whose first answer was given by the Ehresmann Theorem (1950) \cite{ehresmann}: {\it If $\dim X=n$, $\dim Y=m$, $n\geq m$, and $f:X\rightarrow Y$ is a $C^\infty$  proper submersion then it is a fibre bundle}. Note that a mapping between finite-dimensional manifolds  is a submersion if and only if it is a Fredholm map without critical points with  $\Index f= n-m\geq 0$.

The proof of Ehresmann runs  as follows: Let $p$ be a smooth curve in $Y$ beginning at $f(x_0)$ for some $x_0\in X$.   In the finite-dimensional context, a {\it horizontal lifting} of $p$ is a path $q$ in $X$ such that $f\circ q = p$ and its tangent vector field is horizontal, namely $\dot{q}(t)\in\Kernel df(q(t))^{\perp}$.  Since $f$ is a submersion by a differential equations argument (this will be explained shortly):
\begin{itemize}
\item[\buletita] there is at most one horizontal lifting of $p$ beginning at $x_0$;
\item[\buletita] A horizontal lifting of $p$ always exist locally. 
\end{itemize}
{\it If the horizontal liftings can be defined in whole $I=[0,1]$} as well ---e.g. if $f$ is proper---  then the set of horizontal subspaces $\{\Kernel df(x)^{\perp}: x\in X\}$ is an Ehresmann connection for $f$ and the map $f$ is a fibre bundle, as outlined below. Note that $f$ is an open map since it is a submersion. If $f$ is proper then it is also surjective. 

At the beginning of the sixties, based on the ideas of Ehresmann, Hermann \cite{hermann} established the following metric condition  for a $C^\infty$ {\it surjective} submersion $f:X\rightarrow Y$ between finite-dimensional Riemannian manifolds to be a fibre bundle, cf. \ref{localisometry}:
\begin{enumerate}[label={\bf (C25)}]
\item\label{hermanncondition} $X$ is complete and  for all $x\in X$, $y=f(x)$, the canonical isomorphism $\widehat{df(x)}:X/{\Kernel df(x)}\rightarrow T_{y}Y$ preserves the inner products defined by the metrics on these spaces ({\it Riemannian submersion}).
\end{enumerate}
By condition \ref{hermanncondition} the horizontal lifting of a segment curve $p$ in $Y$ has the same length. The usual process of continuation runs into no obstruction since the local liftings always lie in a fixed and bounded, thus a compact region of $X$. Note that the linear projection $p(x,y)=x$ is a trivial example of a non-proper Fredholm map of positive index satisfying the Hermann's conditions for the Euclidean metric in $X=\mathbb{R}^2$ and $Y=\mathbb{R}$. See also \cite{rinehart} and \cite{wolf} for some related results published slightly after.

The properness condition, the Hermann condition, and the ideas in the Ehresmann proof, remain all in the same spirit as in the previous sections. Our goal now is to connect all of the above conditions for split submersions between Banach or Finsler manifolds. To this end, in the next section we introduce the ideas behind the Eells-Earle Theorem.

\subsection{The Earle-Eells Theorem}

The Ehresmann Theorem has been widely reported in the literature, but  the works do not address extensions to the infinite-dimensional setting. An important exception is an article by Earle and Eells \cite{eells} published in the late sixties. They consider a  split submersion map $f:X\rightarrow Y$   between Finsler manifolds and a {\it locally Lipschitz right inverse} of $df$ namely, bundle maps $$s:f^{*}(TY)\rightarrow TX$$  such that for every $x\in X$,
\begin{enumerate}
\item[\buletita] $s(x):T_{f(x)}Y\rightarrow T_xX$ is continuous and linear;
\item[\buletita] for every charts $(W,\varphi)$ at $x$  and $(U,\psi)$ at $f(x)$ with $U\subset f(W)$, the map $d\varphi(\varphi^{-1}(\cdot))s(\varphi^{-1}(\cdot))d\psi(f(\varphi^{-1}(\cdot)))^{-1}$  is locally Lipchitz on $\varphi(W)$.

\item[\buletita] $df(x)s(x)$ is the identity map on $T_{f(x)}Y$.
\end{enumerate}
The symbol  $f^*(TY)$ denotes the vector bundle over $X$ obtained by pulling back $TY$ via $f$. {\it The minimum smoothness required for all results in this section is $C^1$ with locally Lipschitz continuous derivative for  the mapping $f$ and $C^2$ for the manifolds $X$ and $Y$}, so the tangent bundles $TX$ and $TY$ are $C^1$ Banach space bundles.

The Earle-Eells Theorem can be stated as follows. Let $f:X\rightarrow Y$ be a {\it surjective} split submersion between  Finsler manifolds (see Remark \ref{rabieronto}). Then $f$ is a fibre bundle provided:
\begin{enumerate}[label={\bf (C26)}]
\item\label{rightinverse} $X$ is complete and there is a locally Lipschitz right inverse $s$ of $df$ such that for each $y\in Y$ there is a number $\alpha>0$ and a neighborhood $V$ of $y$ such that $\|s(x)\|^{-1}\geq \alpha$ for all $x\in f^{-1}(V)$.
\end{enumerate}
As  shown in the sketch of the proof presented below, to get a fibre bundle it is enough to lift the straight line segments $p_z$ joining $z=f(x)$ to $y$ relative to a chart $(V_y,\psi)$ centered at $y$ where $\psi(V_y)$ is an open ball centered at $0=\psi(y)$, for all $x\in f^{-1}(V_y)$. Also, the liftings $q_x$ must vary continuously and the correspondence $x\mapsto q_x$ must be well defined. So, for every $w\in \psi(V_y)$ we can consider the initial value problem:
\begin{eqnarray*}
\dot q(t)  &= & \xi(q(t))\\
q(0) &=& x
\end{eqnarray*}
where $\xi:X\rightarrow TX$ is a vector field on $X$ defined by $\xi(x)=s(x)[d\psi(f(x))]^{-1}w$. If $s$ is a locally Lipschitz right inverse of $df$ then $\xi$  is a locally Lipschitz vector field on $X$. Therefore for each $x$ in $f^{-1}(V_y)$ and $w$ in $\psi(V_y)$ the above equation has a unique solution $q_x(t)$ in $f^{-1}(V_y)$  defined on an open interval $J$ containing $0$ \cite[p. 116]{palais}. This solution defines  a unique horizontal  path relative to $s$ which is a local lifting of the path $p_z(t)=\psi^{-1}(\psi(z)+tw)$ defined on a maximal domain $[0,\epsilon)$. Recall, according to Earle and Eells, in an infinite-dimensional context a  path $q$ in $X$ is called {\it horizontal relative to} $s$ if $\dot q(t)$ belongs to the image space $s(q(t))T_{f(q(t))}Y$. In particular, for $w=-\psi(z)$ we get a suitable local lifting for $p_z$. 

If $s$ is locally bounded  over $Y$, as  is required by condition \ref{rightinverse}, then it is easy to check that $\ell(q_x)<\infty$ and by completeness of $X$ the path $q_x$ can be defined in whole  $I=[0,1]$ \cite[p. 27]{eells}. Actually,  Earle and Eells  reason by contradiction: if $\epsilon<1$ then there is a Cauchy sequence $\{q(t_n)\}$ converging to some point in $X$ and this implies that the domain of $q$ can be extended in a smooth  manner to an open interval in $I$ containing $[0,\varepsilon)$, thus contradicting the maximality property of  $\epsilon$. This argument inevitably leads us to think of the continuation property. We shall return to this point in the next subsection.

\

\noindent{\sl Sketch of the proof of Earle-Eells Theorem}. Let $y\in Y$ and let $(V_y,\psi)$ be a  chart centered at $y$ such that $\psi(V_y)$ is an open ball in a Banach space centered at $0=\psi(y)$. For every $z\in V_y$ there exists a unique straight line segment $p_z$ relative to $\psi$ joining $z$ to $y$. For any $x\in f^{-1}(V_y)$ ($f^{-1}(V_y)\neq\emptyset$ because $f$ is surjective) consider the line path $p_{f(x)}$ in $V_y$ and the horizontal lift (relative to $s$) $q_x$ starting at $x$ and ending in $f^{-1}(y)$. Then the mapping $\Phi_V:f^{-1}(V_y)\rightarrow V_y\times f^{-1}(y)$ defined by $$\Phi_{V_y}(x)=(f(x), q_x(1))$$ is the desired homeomorphism. The bijection of  $\Phi_V$ and the continuity of $\Phi_V$ and $\Phi_V^{-1}$ follows from the fact that $p_{f(x)}$ is a line segment relative to $\psi$ and from the continuity of the solutions of the corresponding differential equation with respect to the initial conditions. For every $z\in V_y$ there is a homeomorphism  between $f^{-1}(z)$ and $f^{-1}(y)$ obtained by mapping  each $u\in f^{-1}(z)$ into the end-point of the horizontal lifting of $p_z$ starting at $u$.

\begin{remark}\label{rabieronto}{\it  The connectedness of $Y$ implies that $f$ is onto.} So far, the idea has been to establish conditions to ensure that the horizontal liftings  exist globally once it has been tested or assumed that $f$ is onto. 
Nevertheless, Rabier proved with an elementary topological argument that  the surjectivity condition on $f$ can be replace by the connectedness of $Y$; see proof of Theorem 4.1 of \cite{rabier}, cf. Remark \ref{conexidadimplicasobre}. 
\end{remark}

A   split submersion $f:X\rightarrow Y$ has a locally Lipschitz right inverse $s$ of $df$ in the following cases: 
\begin{enumerate}
\item[\buletita] If $X$ and $Y$ are Hilbert spaces then there is a canonical right inverse $s$. Actually, if $y=f(x)$ we can set $s(x):T_yY\mapsto \Kernel df(x)^\perp$ as the inverse of  $df(x)|_{\Kernel df(x)^\perp}$ given by an explicit formula $df(x)^*[d(x)df(x)^*]^{-1}$. In particular, if $df$ is locally Lipschitz, so is $s$; see Lemma 2.5 of \cite{plastockprojections}. 
\item[\buletita]If $X$ and $Y$ are Banach spaces and $df$ is locally Lipschitz then a locally Lipschitz right inverse $s$ can be constructed  by means of a locally Lipschitz partition of unity; see Lemma 2.6 of \cite{plastockprojections}. The same kind of construction can be used to extend this result  for $X$ and $Y$ Finsler manifolds of class $C^{2}$; see Lemma 3B of \cite{eells}. See also Proposition 2.1 of \cite{rabier} for an explicit construction when $Y$ is a Banach space. 
\item[\buletita] If $f$ is a local diffeomorphism then there is only one right inverse given by $s(x)=df(x)^{-1}$ and any local lifting of a $C^1$ path is horizontal relative to $s$. Besides, condition \ref{rightinverse} coincides with the Earle-Eells condition, hence the name. 
\end{enumerate}

\subsection{Topological conditions}

Let $f:X\rightarrow Y$ be a split submersion between Banach manifolds  with locally Lipschitz right inverse for $df$. Assume  $Y$ is  connected. On the one hand, a simple adjustment in the proof of Earle-Eells Theorem shows that {\it if $f$ has the continuation property for the set of all $C^1$ paths then $f$ is a fibre bundle}; see Theorem 2.3 of \cite{oliviajesusfibrations}. For example, a weakly proper map  has the continuation property for $C^1$ paths. If $Y$ is a Banach space, as before, we can restrict the continuation property for lines (condition $L$). This fact was basically noted by Plastock; see Theorem 2.9 of \cite{plastockprojections}. On the other hand, every fibre bundle has the path lifting property \cite[pp. 92, 96]{spanier}. Therefore, each of the following statements implies the next:
 \begin{enumerate}
\item[\buletita] $f$ has the continuation property for the set of $C^1$ paths.
\item[\buletita] $f$ is a fibre bundle.
\item[\buletita] $f$ has the path lifting property.
\end{enumerate}
We have pointed out before that, if $f$ is a local diffeomorphism then all three conditions are equivalent, but this can't be extended to this context. For example, consider the linear projection $\pi(x,y)=x$. The path $p(t)=t$ in $\mathbb{R}$ has local lift $q(t)=\left(t,\frac{1}{t-1}\right)$ defined on $[0,1)$ but there is no  sequence $t_n\rightarrow 1$ such that $q(t_n)$ converges in $\mathbb{R}$. This also shows that the continuation property is not appropriate in important situations. However, the good news is that you only need to apply the monodromy argument to the horizontal liftings corresponding to the line segments relative to a chart,  as we exemplify in the next paragraph. Note that no Finsler structure is needed here.

Now, suppose that $f$ satisfies condition \ref{browder} (Browder). Then for every $y\in Y$ we can choose a chart $(V_y,\psi)$ centered at $y$ such that $f$ is a closed mapping on each component of $f^{-1}(V_y)$ into $V_y$. The integral curves $q_x(t)$
of the initial value problem defined on a maximal interval  $[0,\epsilon)$ considered by Earle and Eells lie in a connected component of $f^{-1}(V_y)$. Let $C$ be the closure of the image of $q_x$. Since $f(C)$ is closed, there exists $x^*\in C$ such that $f(x^*)=p_z(\epsilon)$. Therefore there is an increasing sequence $\{t_n\}$ in $[0,\epsilon)$ convergent to some $t^*$ such that $p_z(t^*)=p_z(\epsilon)$, so $t^*=\epsilon$. Then the path $q_x$ can be extended outside $[0,\epsilon)$ contradicting its maximality. Therefore, $q_x$ can be extended to $I=[0,1]$. So, {\it Browder condition implies that $f$ is a fibre bundle.} The above argument can be used to prove that {\it if $f$ is a closed map then it is a fibre bundle}.

If $f:X\rightarrow Y$ is a proper submersion map between connected Banach manifolds then it is closed surjective map. Let $x\in X$ and $y=f(x)$ thus $f^{-1}(y)$ is a compact submanifold of $X$, hence a finite-dimensional submanifold of $X$ such that $T_xf^{-1}(y)=\Kernel df(x).$ The connectedness of $X$ implies that $\dim \Kernel df(x)=k$ for all $x\in X$ and for some integer $k\geq 0$. Therefore $f$ is Fredholm of nonnegative index. Now, if $f$ is a closed Fredholm map of nonnegative index with locally Lipschitz right inverse for $df$ then it is a fibre bundle such that $\mathcal F$ is a compact  submanifold of $X$ of dimension $\Index f$; see proof of Corollary 2.9 of \cite{oliviajesusfibrations}. Finally, every fibre bundle with compact fibre is a proper map. So, if $f:X\rightarrow Y$ is a  submersion between connected Banach manifolds  with locally Lipschitz right inverse for $df$ then the following statements are equivalent:
 \begin{enumerate}
\item[\buletita] $f$ is a proper map.
\item[\buletita] $f$ is a closed Fredholm map of nonnegative index.
\item[\buletita] $f$ is a fibre bundle with compact fibre $\mathcal F$.
\end{enumerate} 
If $X$ and $Y$ are Banach spaces then a proper Fredholm map of positive index with locally Lipschitz right inverse for $df$ must have a singularity \cite{plastockberger};  see also Proposition 3.1 of \cite{plastockprojections}. In fact,  if $X$ and $Y$ are contractible then $f$ is a closed map if and only if $f$ is a proper map if and only if $f$ is a homeomorphism;  see proof of Corollay 2.10 of \cite{oliviajesusfibrations}. This makes clear the limitations of the properness (or closedness) condition, especially in infinite dimension where even Banach spheres are contractible.

\subsection{Metric conditions via  surjectivity indicator}
Let $f:X\rightarrow Y$ be a split submersion between Finsler manifolds.
In view of the metric conditions stated before for local diffeomorphisms, it is natural to ask: What is the relationship between $\|s(x)\|$ and the surjectivity and injectivity indicators for split submersions? First, note that if $df(x)$ has a nontrivial kernel then $\inj df(x)=0$. So the injectivity indicator of $df(x)$, as expected, remains left out of the running. Suppose that $T:X\rightarrow Y$ is a surjective linear map between Banach spaces. Consider the canonical isomorphism $\hat T:X/\Kernel T\rightarrow Y$. It holds that \cite[Rmk. 4.1]{oliviajesusfibrations}: $$\sur T=\sur \hat T=\inj\hat T= \|\hat T^{-1}\|^{-1}.$$ Thus, the Riemannian condition \ref{hermanncondition} means that $\sur df(x)=1$ for all $x\in X$. So, it would only seem logical to ask whether this condition ---and more generally, all metric conditions given before in terms of the surjectivity indicator--- can be carried on in order to get fibre bundles between Finsler manifolds: the answer is yes, provided $f$ has uniformly split kernels.

A map $f:X\rightarrow Y$ between Finsler manifolds is said to have {\it uniformly split kernels} if there is a constant $c>0$ such that for each $x\in X$ there is a projection $P_x\in L(T_xX)$ with $\Kernel P_x=\Kernel df(x)$ and $\|P_x\|_x\leq c$. This concept was introduced by Rabier \cite{rabier} in the late nineties. A map $f$ has uniformly split kernels, for example, if $X$ is Riemannian, $Y$ is finite-dimensional or if $f$ is a Fredholm submersion of nonnegative index; see Lemma 4.2 and Proposition 3.1 of \cite{rabier}. For submersions with uniformly split kernels with a locally Lipschitz derivative between $C^{2}$ connected Finsler manifolds exist a locally Lipschitz right inverse $s$ of $df$ and a constant $c>0$ such that for every $x\in X$:
\begin{equation}\label{constantrabier}\sur df(x)\leq c\|s(x)\|^{-1}.\end{equation}
So, we can consider that for mappings with uniformly split kernels a cleaner version of \ref{rightinverse}, namely, for each $y\in Y$ there is a number $\alpha>0$ and a neighborhood $V$ of $y$ such that $\sur df(x)\geq \alpha$ for all $x\in f^{-1}(V)$. But, as before, this is only a different way to state the condition \ref{strongsubmersion} when $X$ is complete. This leads to the Rabier Theorem 4.1 of \cite{rabier}: {\it If $f:X\rightarrow Y$ is a  strong submersion with uniformly split kernels and locally Lipschitz derivative between $C^{2}$ connected Finsler manifolds and $X$  is complete then $f$ is a fibre bundle}. Actually, by the arguments given above, we can replace the strong submersion condition in the last sentence in italics by  a strong submersion condition with continuous weight (condition \ref{strongsubmersionweighted}); see Remark 4.4 of \cite{rabier} or even nonincreasing weight; see Lemma 3.1 of \cite{oliviajesusfibrations}. So, the global inversion conditions stated before in terms of $\mu(x)$ can be carried on in this setting, replacing $\mu(x)$ by $\sur df(x)$. For example, if  $f$ satisfies the {\it Katriel, Ioffe, or Hadamard integral condition} (see \cite[Ex. 4.6]{oliviajesusfibrations} for an example of a map satisfying the Hadamard integral condition but which is not a strong submersion). The hypothesis in the Rabier Theorem can be weakened and we can consider the submersion  $f:U\rightarrow Y$ with $U$ open subset of $X$ such that there is no sequence $\{x_n\}$ from $U$, converging to a point on $\partial U$, and such that $f(x_n)$ converges to a point in $Y$, as he indeed pointed out. If $Y$ is contractible, there is a submanifold of $\mathcal F$   of $X$ (the fibre of $f$) and a homeomorphism $\Phi:\mathcal F\times Y\rightarrow X$ such that $f(\Phi(x,y))=y$ for all $x\in \mathcal F$ and $y\in Y$. 
An additional smoothness of $\Phi$ can be established if the Banach space model of $X$ admits a smooth enough partition of unity.

Finally, we propose an extension of property \eqref{janjohansson}. Assume $X$ is a complete connected $C^{2}$ Finsler manifold,  $F$ is a Banach space, and  $f:X\rightarrow F$ is a submersion with uniformly split kernels with a locally Lipschitz derivative.  For each $\rho\geq 0$ let $$\eta(\rho)=\frac{1}{c}\inf_{d_X(x_0,x)\leq \rho}\sur df(x)$$ 
where $x_0\in X$ and $c$ is the constant satisfying \eqref{constantrabier}. Given $r>0$ set  $\varrho=\int_{0}^r \eta(\rho) d\rho$. Theorem \ref{globalgraves} asserts that:
$$\varrho>0\mbox{ implies }B_\varrho(f(x_0))\subset f(B_r(x_0)).$$ We have the following observations:
\begin{enumerate}
\item[\buletita] Because $f$ is a submersion, $df(x_0)$ is onto, so $\sur df(x_0)>0$. Also, the function  $x\mapsto \sur df(x)$  is continuous; see Remark 2.1 of \cite{rabier}. Then there is a $\alpha'>0$ and $r>0$ such that $\sur df(x)>\alpha'$ for all $x\in B_{r}(x_0)$. Therefore if $\alpha=\frac{\alpha'}{c}$ then we have: $$B_{\alpha r}(f(x_0))\subset f(B_r(x_0)).$$
This make a connection with the conclusion of the Graves Theorem; see for instance Theorem 1.2 of \cite{dontchev}.
\item[\buletita] The above inclusion implies that  the Ioffe surjection constant  of $f$ at $x_0$ is positive since $\mbox{sur}(f,x_0)\geq\alpha>0$. 
\item[\buletita] If $\lim_{r\rightarrow\infty}\varrho(r)=\infty$, that is, if $f$ satisfies the Hadamard integral condition, then there is a submanifold of $\mathcal F$   of $X$  and a homeomorphism $\Phi:\mathcal F\times F\rightarrow X$ such that for all  $y\in Y$ the solutions of the equation $y=f(u)$ are of the form $u=\Phi(x,y)$ for each $x\in\mathcal F$. Furthermore if $|y-f(x_0)|\leq\varrho(r)$ then $d_X(u,x_0)<r$.
\item[\buletita] If $X$ is Riemannian then  $c=1$ \cite[p. 656]{rabier}. Suppose also that $F$ is a Hilbert space and $f$ is a Riemannian submersion (condition \ref{localisometry}). Then $\eta(\rho)=1$ for all $\rho>0$ and $\varrho(r)=r$. So $B_r(f(x_0))\subset f(B_r(x_0))$ for any $x_0\in X$. Actually, it is easy to see that $f(B_r(x_0))\subset B_r(f(x_0))$  since the canonical isomorphism $\widehat{df(x)}:X/{\Kernel df(x)}\rightarrow T_{y}Y$ preserves the inner products. Therefore $$B_r(f(x_0))=f(B_r(x_0)).$$ As expected, $f$ is a {\it submetry}. Just to complete the picture, it remains to say that, at least in the finite-dimensional context, every submetry between Riemannian manifold is a Riemannian submersion; see Theorem A of \cite{berestovskiiluis}. 
\end{enumerate}

I would like to thank the referee for the careful review of the previous versions of this paper.


\appendix

\section{Proofs and extra remarks}

\begin{remark}\label{conexidadimplicasobre}
{\it The connectedness of  $Y$ implies that $f$ is onto: another proof} (inspired by the second half of the proof of Theorem 4.1 of \cite{rabier}).  Let $f:X\rightarrow Y$ be a map between Banach manifolds. For every $y\in Y$, let $V_y$ be a domain of a  chart of $Y$ centered at $y$ and $U_y=f^{-1}(V_y)$. {\it Assume that} $U_y\neq\emptyset$ and {\it every line segment relative to a chart can be lifted}. Therefore $f|_{U_y}:U_y\rightarrow V_y$ is onto.  Consider the set $\bar Y=\{y\in Y: f^{-1}(V_y)\neq\emptyset\}.$
The set $\bar Y$ is not empty since, in fact, the set $f(X)$ is contained in $\bar Y$. Furthermore $\bar Y$ is open since $f|_{U_y}$ is onto for $y\in\bar Y$. Now let $y$ be in the boundary of $\bar Y$ such that $V_y\cap\bar Y\neq\emptyset$ and let $z\in V_y\cap\bar Y$. Then $f|_{U_z}:U_z\rightarrow V_z$ is onto. Thus, there is a $x\in U_z$ such that $f(x)=z$. On the other hand, $z\in V_y$ implies that $x\in f^{-1}(V_y)$ 
whereby $f^{-1}(V_y)\neq\emptyset$ and therefore $y\in\bar Y$. So $\bar Y$ is also closed in $Y$. By connectedness of $Y$ we have that $\bar Y=Y$ hence $f$ is onto. This reasoning is important because it implies that {\it we just need to lift the paths $p_z$ for $z$ close to $y\in f(X)$ when $Y$ is connected}. 
\end{remark}

\begin{lemma}\label{cartan-hadamardimages} The set {\it $S_{y_0}$ is open and the mapping $f_x^{-1}$ is an inverse of $f$ with domain $S_{y_0}$} \end{lemma}

\begin{proof}
Let $y_0\in Y$ and let $p_z$ be the unique minimizing geodesic segment joining $y_0$ to $z$ in $Y$ (Corollary 1.12 of \cite{neeb}). Let $S_{y_0}$ be the star with vertex $y_0$ defined as the set of all $z\in Y$ for which there is a lifting $q_z$ of $p_z$ such that $q_z(0)=x$. Let $f_x^{-1}(z):=q_z(1)$ for $z\in S_{y_0}$ where $q_z$ is the lifting of $p_z$. Let $d_Y$ be the Finsler distance of $Y$. For all $u$ in the image of  $q_z$ there is an open neighborhood  $U^u$ and $r_u>0$ such that  $f|_{U^u}:U^u\rightarrow B_{r_u}(f(u))$ is a homeomorphism. Let $V^u\subset U^u$ be an open set such that $f(V^u)=B_{\frac{r_u}{2}}(f(u))$. For compactness and connectedness of the image of  $q_z$, there are $u_1,\dots, u_m$ in the image of  $q_z$ such that $q_z\subset\bigcup_{k=1}^m V^{u_k}$. Furtheremore, $V^{u_i}\cap V^{u_j}\neq\emptyset$ if and only if $|i-j|\leq 1$. Also $x\in V^{u_1}$ and  $q_z(1)\in V^{u_m}$. For $k=1,\dots,m$, let $V_k=V^{u_k}$, $U_k=U^{u_k}$, $z_k=f(u_k)$, $r_k=r_{u_k}$, and $B_k=B_{\frac{r_k}{2}}(z_k)$. Therefore,  $\bigcup_{k=1}^m V_k$ is an open covering of the image of  $q_z$ and $\bigcup_{k=1}^m B_k$ is an open covering of $p_z$. Let $s_k:B_{r_k}\rightarrow U_k$ be the inverse of $f|_{U_k}$. Let $0=t_0<t_1<\dots<t_m=1$ be a partition of $I$ such that for $k=1,\dots,m$, $q_z[t_{k-1},t_k]\subset V_k$. For $j=1,\dots,m-1$ let $\tilde u_j=q_z(t_j)\in V_j\cap V_{j+1}=\tilde V_j$. The set $\tilde V_j$ is open hence $f(\tilde V_j)$ is open in $Y$ and contains $\tilde z_j=p_z(t_j)$. Furthermore, $f|_{\tilde V_j}:\tilde V_j\rightarrow f(\tilde V_j)$ is a homeomorphism and $s_j$ coincides exactly with $s_{j+1}$ on $f(\tilde V_j)$. Let $\delta_j>0$ such that $B_{\delta_j}(\tilde z_j)\subset f(\tilde V_j)$ and $\epsilon>0$ such that
$$\textstyle 0<\epsilon<{\rm dist}\hspace{0.03in}\left({\rm image}\hspace{0.03in}p_z;Y\setminus\bigcup_{k=1}^m B_k\right),$$ and 
$0<\epsilon<\min\{r_1,\dots,r_m,\delta_1,\dots,\delta_m\}$. Therefore, by continuity of $(t,z)\mapsto p_z(t)$ (Theorem 2.6  of \cite{oliviajesuscoverings}) there is $\delta>0$ such that if $w\in B_\delta(z)$ and $B_\delta(z)\subset B_{r_m}(z_m)$ then $d_Y(p_w(t),p_z(t))\leq\frac{\epsilon}{2}$ for all $t\in I$. For $j=1,\dots,m-1$, if $t\in[t_{j-1},t_j]$ then $d_Y(p_z(t),z_j)<\frac{r_j}{2}$. Therefore,
$$d_Y(p_w(t),z_j)\leq d_Y(p_w(t),p_z(t))+d_Y(p_z(t),z_j)<r_j.$$ So, $p_w[t_{j-1},t_j]\subset B_{r_j}(z_j)$ where the local inverse $s_j$ is defined. On the other hand, $d_Y(p_w(t_j),\tilde z_j)<\frac{\epsilon}{2}<\delta_j$ thus $p_w(t_j)\in B_{\delta_j}(\tilde z_j)$. Therefore, $s_j(p_w(t_j))$ is equal to $s_{j+1}(p_w(t_j))$. In conclusion, the path $q_w$ defined by $s_k\circ p_w$ in each piece $[t_{k-1},t_k]$ for $k=1,\dots,m$ is well defined and is a lifting of  $p_w$ such that $q_w(0)=x$. Finally, $B_\delta(z)\subset S_{z_0}$. Since $f_x^{-1}$ coincides with  $s_m$ in $B_\delta(z)$ then it is continuous in $z$.
\end{proof}

\begin{lemma}\label{riemanniancase}
Let $f:X\rightarrow Y$  be a local homeomorphism between Banach manifolds. Assume $Y$ is Riemannian and connected. If $f$ has the continuation property for the set of minimal geodesics in $Y$ then $f$ is a covering map.
\end{lemma}

\begin{proof}Let $X$  be a Banach manifold and $(Y,g)$ be Riemannian manifold. For every $y\in Y$ there exists $r$ sufficienty small such that $\exp_y:B_g(0,r)\rightarrow B_g(y,r)$ is a diffeomorphism, where $B_g(0,r)$ and $B_g(y,r)$  are the open ball of radius $r$ centered at $0$ in $T_yY$ and at $y$ in $Y$, respectively.  Then every $z\in B_g(y,r)$ can be joined by a unique minimal geodesic $p_z$ in $V$ ---namely, $\exp_y(tv)$ for  $v=\exp_y^{-1}(z)$--- \cite[pp. 222--227]{lang} and  the map $(t,z)\mapsto p_z(t)$ is continuous.   Let $V_y=B_g(y,r)$. As in Remark \ref{conexidadimplicasobre}, the connectedness of $Y$ implies that the set $\bar Y=\{y\in Y: f^{-1}(V_y)\neq\emptyset\}$ is whole $Y$. Thus $f$ is onto. Finally, by the second part of the proof of Theorem 2.6 of \cite{oliviajesuscoverings}, the continuity of the map $(t,z)\mapsto p_z(t)$ implies that the sets $O_u=\{q_z(1):z\in V\}$ with $u\in f^{-1}(y)$ form the desired disjoint family of open sets. 
\end{proof}

\begin{remark}
If $Y$ is finite-dimensional and complete we have a simpler proof since, by the Hopf-Rinow Theorem, any two points can be joined by a minimal geodesic. This argument can not be applied in general since the Hopf-Rinow theorem fails in an infinite dimension, even more, there exists a  complete infinite-dimensional Riemannian manifold and  two points on there that cannot be joined by any geodesic at all \cite{atkin}. 
\end{remark}

\begin{theorem}\label{firstcharacterization}
If $X$ and $Y$ are Finsler manifolds, $Y$ is connected, and $f:X\rightarrow Y$ is a Fredholm map of index zero then $f$  is a smooth covering map provided $f$ satisfies the Earle-Eells condition. 
\end{theorem}

\begin{proof}
Let $y\in f(X)$. By \ref{earleell} there exists $\alpha>0$ and a neighborhood $V$ of $y$ such that  $\inj df(x)\geq\alpha$ for all $x\in f^{-1}(V)$. Without loss of generality, we can assume that $V$ is the domain of a chart centered at $y$.  Let $p$ be a line segment relative to this chart and let $q$ be a local lifting of $p$ defined on $[0,\varepsilon)$ starting at some point $x_0\in f^{-1}(V)$. Since the image of the path $q$ is contained in $f^{-1}(V)$ and  $f$ is a local diffeomorphism then $\|df(x)^{-1}\|^{-1}=\inj df(x)\geq \alpha$ for all $x$ in the image of $q$. Therefore $\ell(q)<\alpha^{-1}\ell(p)<\infty$. So, for every sequence $\{t_n\}$ in $[0,\varepsilon)$ converging to $\varepsilon$ it is easy to see that $\{q(t_n)\}$ is a Cauchy sequence in $X$. So $\{q(t_n)\}$  converges in $X$; see Lemma 5.1 of \cite{oliviajesuscoverings}. Therefore $q$ can be extended to whole $I=[0,1]$; see proof of Theorem 2.6 of  \cite{oliviajesuscoverings}. We find that $f$ is a smooth covering map taking into account Remark \ref{conexidadimplicasobre}.  The same argument holds if we use the surjectivity indicator instead of the injectivity indicator.
\end{proof}

\begin{lemma}\label{lemmafinitefibre}
If $f$ is a smooth covering map with finite fibre then it satisfies the Earle-Eells condition.
\end{lemma}

\begin{proof} 
For all $y\in Y$ there exists $W$ of $y$ such that $f^{-1}(W)$ is the union of a disjoint family of open sets $\{O_{x_1},\dots O_{x_n}\}$ of $X$, each of which is mapped diffeomorphically onto $W$ by $f$ and $f^{-1}(y)=\{x_1,\dots, x_n\}$.  Since $f$ is a local diffeomorphism and $x\mapsto \|df(x)^{-1}\|$ is continuous (see also Theorem 2.7 of \cite{palais})  for every $i=1,\dots, n$ there is a neighborhood $U_{i}\subset O_{x_i}$ of $x_i$ and $\alpha_i>0$ such that $\mu(x)=\|df(x)^{-1}\|^{-1}\geq \alpha_i$ for all $x\in U_i$. Let $\alpha=\min\{\alpha_1,\dots,\alpha_n\}$ and $V=\cap_{i=1}^n f(U_{i})$. Let $u\in f^{-1}(\cap_{i=1}^n f(U_{i}))$ and let $y=f(u)\in W$. The fibre of $y$ is contained in the disjoint union of open sets $\cup_{i=1}^n O_{x_i}$, then $u$ is in some $O_{x_j}$.  Suppose that $u\notin U_j$ thus $f(u)\notin f(U_j)$ since $f$ is injective in $O_{x_j}$. Therefore, $f(u)\notin\cap_{i=1}^n f(U_{i})$ and we get a contradiction. So, $f^{-1}(V)\subset U_j$ and $\mu(x)\geq \alpha_j\geq \alpha$ for all $x\in f^{-1}(V)$.
\end{proof}

\begin{lemma}\label{hadamardimpliesplastock}
Let $f:X\rightarrow Y$ be a Fredholm map of index zero between  connected Finsler manifolds satisfying {\rm \ref{hadamardintegral}}. Then:
\begin{enumerate}
\item[\buletita] $f$ is a smooth covering map.
\item[\buletita] $Y$ is complete.
\end{enumerate}
\end{lemma}

\begin{proof}
Let $f$ be a Fredholm map of index zero satisfying the Hadamard integral condition. For the first two statements we can proceed as in the proof of Corollary 7 of \cite{oliviajesusisabel}. Only a sketch is given. If $p$ is a rectifiable path in $Y$ starting at $f(x_0)$ there exists $r>0$ such that $\ell(p)<\varrho(r)$. So $p$ can be lifted. Since every point can be joined to $f(x_0)$ by a rectifiable path then every rectifiable path in $Y$ can be lifted. Thus  $f$ is a smooth covering map.

Now, let $\{y_n\}$ be a Cauchy sequence in  $Y$. Let $\sigma_n$ be a path from $y_n$ to $y_{n+1}$ and $\sigma_0$ a path from $f(x_0)$ to $y_1$. Without loss of generality we can suppose that $\ell(\sigma_n)<2^{n}$. Now, for each $n\geq 1$ consider the path $p_n$ which is the concatenation of $\sigma_0, \sigma_1,\dots,\sigma_n$. Then $\ell(p_n)<d_Y(f(x_0),y_1)+2$. Let $r>0$ such that $$\varrho(r)>d_Y(f(x_0),y_1)+2.$$ Then $\ell(p_n)<\varrho$. Each $p_n$ can be lifted to a path $q_n$ contained in the ball with radius $r$ centered at $x_0$ such that  $q_n(0)=x_0$. Let $\alpha=\inf_{d_X(x_0,x)\leq r}\mu(x)>0$. If $\gamma_n$ is the restriction of $q_n$ such that $f(\gamma_n)=\sigma_n$,  by the chain rule $\ell(\gamma_n)\leq \alpha^{-1}\ell(\sigma_n)$. If $x_n=\gamma_n(0)$ for $n\geq 1$ then $d_X(x_n,x_{n+1})\leq\alpha^{-1} 2^{-n}$. Thus $\{x_n\}$ is a Cauchy sequence in $X$ and is therefore convergent, so $\{y_n\}$ is also convergent.
\end{proof}

\begin{lemma}\label{lemaconectionkatrielps}
Let $f:X\rightarrow Y$ be a Fredholm map of index $0$ between Hilbert spaces. Then the Katriel condition implies  that $F_y$  satisfies the  $\mbox{PS}$-condition for all $y\in Y$.
\end{lemma}

\begin{proof}
It is well known that the PS-condition is equivalent to the $\mbox{PS}_c$-condition for any real $c$. 
Since $F_y$ is a non-negative function, it is enough to prove that for all $y\in Y$, $F_y$ satisfies the $\mbox{PS}_c$-condition for any real $c\geq 0$. We shall prove first that {\it there are no $\mbox{PS}_c$-sequences for $F_y$ with $c>0$}. Actually, suppose that there is a $\mbox{PS}_c$-sequence $\{x_n\}$ for $F_y$ with $c>0$. Then there exists $\varrho>0$ such that $F_y(x_n)<\varrho$. Since $f$ satisfies the Katriel condition then  $\inf\{\sur f(x):F_y(x)<\varrho\}>0$. So $\sur f(x_n)\geq\alpha$ for all $n$ and some $\alpha>0$. Therefore $$|f(x_n)-y|<\alpha^{-1}|\nabla F_y(x_n)|=\alpha^{-1}|df(x_n)^*(f(x_n)-y)|.$$ Thus $F_y(x_n)\rightarrow 0$, so  we get a contradiction. Finally, it is easy to see that {\it every $\mbox{PS}_0$-sequence for $F_y$ converges trivially}. If $\{x_n\}$ is a  $\mbox{PS}_0$-sequence  then $f(x_n)\rightarrow y$. Since the Katriel condition implies that $f$ is a global diffeomorphism then $x_n\rightarrow f^{-1}y$.
\end{proof}

\begin{theorem}\label{globalgraves}
Let $X$ be a $C^{2}$ complete and connected Finsler manifold, let $F$ be a Banach space, and  let $f:X\rightarrow F$ be a  submersion with uniformly split kernels and a locally Lipschitz derivative. Let $c$ be the constant satisfying \eqref{constantrabier} and $x_0\in X$. For each $\rho>0$ let $$\eta(\rho)=\frac{1}{c}\inf_{d_X(x_0,x)\leq \rho}\sur df(x).$$ Given $r>0$ set $\varrho(r)=\int_{0}^r \eta(\rho) d\rho.$ If $\varrho=\varrho(r)>0$ then
$$B_\varrho(f(x_0))\subset f(B_r(x_0)).$$
\end{theorem}

\begin{proof}
Let $s$ be a locally Lipschitz right inverse of $df$ satisfying \eqref{constantrabier}. For  $w\in F$ the mapping $s_w(\cdot)=s(\cdot)w$ is a locally Lipchitz section. Therefore for every $x\in X$ there is a unique semi-flow $q(t,x, w)$ characterized by:
\begin{eqnarray*}
\frac{\partial q}{\partial t}(t,x,w) &=& s(q(t,x,w))w,\\
q(0,x,w) &=& x.
\end{eqnarray*}
and defined over a maximal interval  $J(x,w)=(a(x,w),b(x,w))\subset\mathbb{R}$ containing $0$ \cite[p. 116]{palais}. The continuous dependence upon parameters $x$ and $w$ implies that the set $$\Omega=\bigcup_{(x,w)\in X\times F}(a(x,w),b(x,w))\times \{(x,w)\}$$ is open in $\mathbb{R}\times X\times F$ and $q:\Omega\rightarrow X$ is continuous.
Furthermore, we have that  $$f(q(t,x,w))=f(x)+tw\mbox{ for all } t\in J(x,w).$$
since $\partial(f\circ q)(t,x,w)/\partial t=df(q(t,x,w))s(q(t,x,w))w=w$.

\

\noindent{\bf Claim:} Retain the hypothesis of Theorem \ref{globalgraves}. Let $0<|w|<\varrho=\varrho(r)$. Then:

\begin{enumerate}
\item $d_X(q(t,x_0,w),x_0)<r$ for all $t\in [0,1)$.
\item $b(x_0,w)>1$.
\end{enumerate}

\

\noindent {\it Proof of Claim}. Case $\eta(r)>0$.  Suppose that (1) is not true. Let $q(t)=q(t,x_0,x)$ and $$\delta=\inf\{t\in [0,1):d_X(q(t),x_0)\geq r\}<1.$$ Note that $d_X(q(\delta),x_0)=r$. Let $\xi(\rho)=\max\{d_X(q(t),x_0):t\in[0,\rho]\}$. The function $\xi$ is continuous and nondecreasing and for every $\rho\in(0,\delta]$ we have that $0<\xi(\rho)\leq r$ and then $0<\eta(\xi(\rho))<\infty$. We claim that if $0\leq\rho'<\rho\leq\delta$ then: 
$$\xi(\rho)-\xi(\rho')\leq\frac{|w|(\rho-\rho')}{\eta(\xi(\rho))}.$$
If $\xi(\rho')=\xi(\rho)$ this inequality is evident. Suppose now that $\xi(\rho')<\xi(\rho)$, there exists some $\rho^*\in(\rho',\rho]$ such that $\xi(\rho)=d_X(q(\rho^*),x_0)$. By \eqref{constantrabier} for fixed $t\in[\rho',\rho^*]$ we have: 
$$\sur df(q(t))\|\dot q(t)\|\leq c\|s(q(t))\|^{-1}\|\dot q(t)\|=c\|s(q(t))\|^{-1}\|s(q(t))w\|\leq c|w|.$$
So,
 $$\inf_{[\rho',\rho^*]}\sur df(q(t))\int_{\rho'}^{\rho^*}\|\dot q(t)\|dt\leq c\int_{\rho'}^{\rho^*}|w|dt.$$
 Then, \begin{equation}\label{meanvaluesubmersions}d(q(\rho'),q(\rho^*))\inf_{[\rho',\rho^*]}\frac{1}{c}\sur df(q(t))\leq |w|(\rho^*-\rho').\end{equation} Therefore: 
 $$d(q(\rho'),q(\rho^*))\cdot \eta(\xi(\rho^*))\leq |w|(\rho^*-\rho')$$
and 
$$\xi(\rho)=d_X(q(\rho^*),x_0)\leq d(q(\rho'),x_0)+\frac{|w|(\rho^*-\rho)}{\eta(\xi(\rho^*))}\leq \xi(\rho')+\frac{|w|(\rho'-\rho)}{\eta(\xi(\rho))}.$$
This establishes the claim. Now, for each partition $0=\rho_0<\rho_1<\cdots \rho_n=\xi(\delta)$, we can find $0=\tau_0<\tau_1<\cdots \tau_n=\delta$ such that $\rho_i=\xi(\tau_i)$ for each $i=0,1,\dots,n$. Then,
$$\sum_{i=1}^n\eta(\rho_i)(\rho_i-\rho_{i-1})=\sum_{i=1}^n\eta(\xi(\tau_i))(\xi(\tau_i)-\xi(\tau_{i-1}))\leq\sum_{i=1}^n|w|(\tau_i-\tau_{i-1})=\delta|w|<\varrho.$$
Therefore, $$\int_{0}^{d_X(q(\delta),x_0)}\eta(\rho)d\rho\leq\int_0^{\xi(\delta)}\eta(\rho)d\rho<\int_0^r\eta(\rho)d\rho.$$
So $d_X(q(\delta),x_0)< r$, which is a contradiction. 

Now suppose by contradiction that $b(x_0,w)\leq 1$.   By (1) for all $t\in [0,b(x_0,w))$:
$$0<\eta(r)<\frac{1}{c}\sur(q(t,x_0,w))$$  Now, if $\{t_n\}$ is an increasing sequence in $[0,b(x_0,w))$ convergent to $b(x_0,w)$, by \eqref{meanvaluesubmersions} we obtain that for $m\geq n$: 
$$d_X(q(t_n),q(t_m))\leq\frac{|w|(t_m-t_n)}{\eta(r)}.$$ Therefore, $\{q(t_n)\}$ is a Cauchy sequence and then is convergent by completeness of $X$ and the definition of the Finsler metric in a Finsler manifold. So $$x=\lim_{t\rightarrow b(x_0,w)^-} q(t,x_0,w)\in X.$$ Therefore $q(t,x_0,w)$ could then be extended to values $t>b(x_0,w)$ contradicting the maximality of $J(x,w)$.

We consider now the case $\eta(r)=0$. Then $r\geq r_0=\sup\{\rho>0:\eta(\rho)>0\}$ and we have that 
$$\varrho=\int_0^r\eta(\rho)d\rho=\int_0^{r_0}\eta(\rho)d\rho.$$ If $0<|w|<\varrho$ we can choose $r'$ and $\varrho'$ such that $\eta(r')>0$, $0<|w|<\rho'<\rho$ and $\varrho(r')=\varrho'$. Then by the previous case we obtain that $d_X(q(t,x_0,w),x_0)<r'<r$, for all $t\in[0,1)$ and $b(x_0,w)>1$.

\

Let $y\in B_\varrho(f(x_0))$. Then there is $w$ with $|w|<\varrho$ such that $y=f(x_0)+w$. By the above arguments the path $p(t)=f(x_0)+tw$ can be lifted to a path 
$q(t,x_0,w)$ in $B_r(x_0)$. In particular $y=f(q(1,x_0,w))\in f(B_r(x_0))$.

\end{proof}
\bibliographystyle{plain}
\bibliography{references}


\end{document}